\documentclass{article}

\usepackage[a4paper]{geometry}
\usepackage{amsmath}
\usepackage{amssymb}
\usepackage{amscd}
\usepackage{latexsym}
\usepackage{graphicx}
\usepackage{mathrsfs}
\usepackage{amsthm}

\usepackage[table]{xcolor}    

\def\R {\mathbb{R}}

\def\Z {\mathbb{Z}}

\newcommand{\curl}{\mathop{\mathrm{curl}}}

\newcommand{\sign}{\mr{sign}}

\newcommand{\lk}{\ell\raisebox{-2pt}{$\!\kappa$}}
\newcommand{\coil}{\mr{Coil}}
\renewcommand{\o}{\mr{o}}

\newtheorem{theo}{Theorem}
\newtheorem{lemma}[theo]{Lemma}
\newtheorem{defin}[theo]{Definition}
\newtheorem{alg}{Algorithm}
\newtheorem{cor}[theo]{Corollary}
\newtheorem{remark}[theo]{Remark}
\newcommand{\mb}{\mathbf}
\newcommand{\bs}{\boldsymbol}
\newcommand{\lra}{\longrightarrow}
\newcommand{\mr}{\mathrm}
\newcommand{\mc}{\mathcal}

\newcommand{\mk}{\mathfrak}

\newcommand{\T}{\mc{T}}
\newcommand{\vv}{\mb{v}}
\newcommand{\E}{\mc{E}}
\newcommand{\F}{\mc{F}}
\newcommand{\G}{\mc{G}}
\newcommand{\A}{\mc{A}}
\newcommand{\B}{\mc{B}}
\newcommand{\K}{\mc{K}}

\newcommand{\sangle}{{\scriptscriptstyle\angle}}
\newcommand{\cb}{\sigma_{_{\!\B'}}}

\begin{document}

\title{Efficient construction of homological Seifert surfaces}

\author{Ana Alonso Rodr\'{i}guez\thanks{Dipartimento di Matematica, Universit\`a di Trento, 38123 Povo (Trento), Italy} \and Enrico Bertolazzi \thanks{Dipartimento di Ingegneria Meccanica e Strutturale, Universit\`a di Trento, 38123 Mesiano (Trento), Italy} \and Riccardo Ghiloni\footnotemark[1] \and Ruben Specogna\thanks{Universit\`a di Udine, Dipartimento di Ingegneria Elettrica, Gestionale e
Meccanica, Via delle Scienze 206, 33100 Udine, Italy}}


\maketitle

\begin{abstract}
Let $\Omega$ be a bounded domain of $\R^3$ whose closure $\overline{\Omega}$ is polyhedral, and let $\T$ be a trian\-gulation of $\overline{\Omega}$. Assuming that the boundary of $\Omega$ is sufficiently regular, we provide an explicit formula for the computation of homological Seifert surfaces of any $1$-boundary $\gamma$ of $\T$; namely, $2$-chains of $\T$ whose boundary is $\gamma$. It is based on the existence of special spanning trees of the complete dual graph of $\T$, and on the computation of certain linking numbers associated with those spanning trees. If the triangulation $\T$ is fine, the explicit formula is too expensive to be used directly. For this reason, making also use of a simple elimination procedure, we devise a fast algorithm for the computation of homological Seifert surfaces. Some numerical experiments illustrate the efficiency of this algorithm.
\end{abstract}


\section{Introduction}

\subsection{The results}

A crucial concept of knot theory is the one of Seifert surface. A Seifert surface of a polygonal knot of $\R^3$ is an orientable nonsingular polyhedral surface of $\R^3$ having the knot as its boundary. This notion has a natural counterpart in homology theory.
Let $\Omega$ be a bounded domain of $\R^3$ whose closure $\overline{\Omega}$ in $\R^3$ is polyhedral, and let $\T$ be a triangulation of $\overline{\Omega}$. We assume that the boundary $\partial\Omega$ of $\Omega$ satisfies a mild regularity condition that we specify at the end of this section.
A $1$-cycle $\gamma$ of $\T$ is a formal linear combination (over integers) of oriented edges of $\T$ with zero boundary. The $1$-cycle $\gamma$ is said to be a $1$-boundary of $\T$ if it is equal to the boundary of a formal linear combination $S$ of oriented faces of $\T$. If such a $S$ exists, we call it homological Seifert surface of $\gamma$ in $\T$.

The identification of homological Seifert surfaces is a fundamental task in very different fields. For example, they appear in Stokes' theorem: given a sufficiently regular vector field ${\bf Z}$ defined in $\overline{\Omega}$ and a $1$-boundary $\gamma$ of $\T$, we have that $\oint_\gamma {\bf Z} \cdot ds = \int_S \curl {\bf Z} \cdot \nu$, where $S$ is any homological Seifert surface of $\gamma$ in $\T$. As a consequence, homological Seifert surfaces are a powerful tool in computational electromagnetism for the construction of discrete vector potentials; namely, vector fields with assigned discrete curl (see, e.g., \cite{Boss98,GK04,AVbook}).

Homological Seifert surfaces are also a key point in the construction of bases of the relative homology group $H_2(\overline  \Omega, \partial \Omega; \Z)$. Let $\{\sigma'_m\}_{m=1}^g$ be $1$-boundaries of $\T$ contained in $\partial\Omega$ whose homo\-logy classes in $\R^3 \setminus \Omega$ forms a basis of the first homology group of $\R^3 \setminus \Omega$. If $S_m$ is a homological Seifert surface of $\sigma'_m$ in $\T$ for each $m \in \{1,\ldots,g\}$, then the Poincar\'e--Lefschetz and the Alexander duality theorems ensure that the relative homology classes $[S_m]$ of the $S_m$'s form a basis
 of $H_2(\overline  \Omega, \partial \Omega; \Z)$.

The problem of constructing homological Seifert surfaces is connected to the more geometric one of finding genuine Seifert surfaces. If the $1$-boundary $\gamma$ of $\T$ is a polygonal knot, then a homological Seifert surface of $\gamma$ determinates a Seifert surface of $\gamma$ if the union of its faces is an orientable nonsingular polyhedral surface of $\R^3$. The homological Seifert surfaces we compute do not have necessarily this regularity. However, we think that, in future investigations, this approach could be taken as the starting point to obtain Seifert surfaces.

Even if the question of computing homological Seifert surfaces is very natural and significant, to the best knowledge of the authors, there are not general and efficient algorithms to compute such surfaces. Given an orientation of the edges and of the faces of the triangulation $\T$ of $\overline{\Omega}$, the problem can be formulated as a linear system with as many unknowns as faces and as many equations as edges of $\T$. The matrix $A$ of this linear system is the incidence matrix between faces and edges of $\T$. This matrix is very sparse because it has just three nonzero entries per columns and the number of nonzero entries on each row is equal to the number of faces incident on the edge corresponding to the row. We are looking for an integer solution of this sparse rectangular linear system. This kind of problems are usually solved using the Smith normal form, a computationally demanding algorithm even in the case of sparse matrices (see e.g. \cite{Mun84}, \cite{DSV01}).

A first difficulty to devise a general and efficient algorithm to compute a homological Seifert surface $S$ of a given $1$-boundary $\gamma$ of $\T$ is that this problem has not a unique solution. Indeed, the kernel of $A$ is never trivial. If $\mk{t}$ is the number of tetrahedra of $\T$ and $\Gamma_0,\Gamma_1,\ldots,\Gamma_p$ are the connected components of $\partial\Omega$, then $\ker(A)$ is a free abelian group of rank $\mk{t}+p \,$; namely, $\ker(A)$ is isomorphic to $\Z^{\mk{t}+p}$. One of its basis is given by the boundaries of tetrahedra of $\T$ and by the $2$-chains $\gamma_1,\ldots,\gamma_p$ associated with the triangulations of $\Gamma_1,\ldots,\Gamma_p$ induced by $\T$. This follows easily from the fact that the third homology group of $\overline{\Omega}$ is null and the $2$-chains $\gamma_1,\ldots,\gamma_p$ represent a basis of the second homology group of $\overline{\Omega}$ (see Remark~\ref{rem:intro} below).

A natural strategy to obtain a unique solution $S$ is to add $\mk{t}+p$ equations, by setting equal to zero the unknowns corresponding to suitable faces $f_1,\ldots,f_{\mk{t}+p}$ of $\T$. From the geometric point of view, this is equivalent to impose that the homological Seifert surface $S$ of $\gamma$ does not contain the faces $f_1,\ldots,f_{\mk{t}+p}$. Now the problem is to understand how to choose such faces. Our idea to make this choice is to use a suitable spanning tree of the dual complex of $\T$.
More precisely, we introduce the \textit{complete dual graph of $\T$} denoted by $\A'$. Let $F$ be the set of faces of $\T$, $F_\partial$ the set of faces of $\T$ contained in $\partial\Omega$ and $E_\partial$ the set of edges of $\T$ contained in $\partial\Omega$. The dual edge $\epsilon'_f$ of a face $f \in F$ and the dual edge $\epsilon'_\ell$ of an edge $\ell \in E_\partial$ are defined in the following way. If $f \in F_\partial$, then it is contained in a unique tetrahedron $t$ and $\epsilon'_f:=\{B(f),B(t)\}$, where $B(f)$ is the barycenter of $f$ and $B(t)$ the barycenter of $t$. If $f$ is an internal face of $\T$ (namely $f \in F \setminus F_\partial$), then it is the common face of exactly two tetrahedra $t_1$ and $t_2$, and $\epsilon'_f:=\{B(t_1),B(t_2)\}$. Similarly, if $\ell \in E_\partial$, then it is the common edge of exactly two faces $f_1,f_2$ in $F_\partial$, and $\epsilon'_f:=\{B(f_1),B(f_2)\}$. The vertices of $\A'$ are the barycenters of tetrahedra of $\T$ and the barycenters of faces in $F_\partial$, and the edges of $\A'$ are the dual edges $\{\epsilon'_f\}_{f \in F}$ and $\{\epsilon'_\ell\}_{\ell \in E_\partial}$.
Let $\B'$ be a spanning tree of $\A'$. Denote by $N_{\B'}$ the number of faces of $\T$ whose dual edge belongs to $\B'$; namely, the number of edges of $\B'$ not contained in $\partial\Omega$. It is not difficult to see that, for all spanning tree $\B'$ of $\A'$, $N_{\B'}\ge \mk{t}+p$. The equality holds true if and only if, for each $i \in \{ 0,1,\dots,p \}$, the graph induced by $\B'$ on $\Gamma_i$ is a spanning tree of the graph induced by $\A'$ on $\Gamma_i$ (see Remark~\ref{rem:intro}). If the spanning tree $\B'$ of $\A'$ has the latter property, then we call it \textit{Seifert dual spanning tree of $\T$} (see Definition~\ref{def:sette}).

Our main result, Theorem \ref{thm:main}, shows that if $\B'$ is a Seifert dual spanning tree, then, for every $1$-boundary $\gamma$ of $\T$, there exists a unique homological Seifert surface $S$ of $\gamma$ in $\T$, which does not contain faces of $\T$ whose dual edges belong to $\B'$. Furthermore, if $f$ is a face of $\T$ whose dual edge $\epsilon'_f$ does not belong to $\B'$, then $f$ appears in $S$ with a coefficient equal to the linking number between $\gamma$ (suitably retracted inside $\overline{\Omega}$) and the unique $1$-cycle $\cb(\epsilon'_f)$ of $\A'$ with all the edges except $\epsilon'_f$ contained in $\B'$.

As a byproduct, in Theorem \ref{thm:internal}, we solve completely the related problem concerning the existence and the construction of internal homological Seifert surfaces of $\gamma$; namely, homological Seifert surfaces of $\gamma$ formed only by internal faces of $\T$.

The construction of Seifert dual spanning trees of $\T$ is quite easy and the computation of the linking number between two simplicial $1$-cycles of $\R^3$ can be performed in a very accurate and efficient way (see \cite{BG12,Arai13}). However, for a fine triangulation $\T$, the number of faces whose dual edge does not belong to a given Seifert dual spanning tree of $\T$ is very large: it is equal to $\mk{e}-\mk{v}+1-g \geq \frac{1}{2}\mk{v}+1-g$, where $\mk{e}$ is the number of edges of $\T$, $\mk{v}$ is the number of vertices of $\T$ and $g$ is the first Betti number of $\overline{\Omega}$ (see Section~\ref{sec:elim_alg}). Thus, the use of the explicit formula in terms of linking number turns to be too expensive. To overcome this difficulty, we adopt an elimination procedure, similar to the one proposed by Webb and Forghani in \cite{WF89} for the solution of three-dimensional magnetostatic problems. When this procedure fails, one can compute a new unknown using the explicit formula and then restart the elimination algorithm.

We remark that what developed in this paper for simplicial complexes extends to general polyhedral cell complexes; namely, finite regular CW complexes.

The remainder of the paper is organized as follows. We conclude this introductory section by precising the weak topological requirements on the domain $\Omega$. In Section \ref{sec:preliminaries}, we recall some classical homological notions and constructions, and we introduce some new geometric concepts, as \emph{corner edge, coil} and \emph{plug}. Section \ref{sec:results} is devoted to the presentation and the proof of our main result (Theorem \ref{thm:main}) and of some of its consequences (Theorem \ref{thm:internal} and Corollary \ref{cor:internal}). In Section \ref{sec:elim_alg}, we describe the above mentioned elimination algorithm to improve the implementation of our main theorem. Finally, in Section \ref{sec:numerical}, we perform several numerical experiments of the algorithm.


\subsection{Topological hypotheses on the domain $\bs{\Omega}$}

The results of this paper are valid on very general domains that we are going to describe. A compact connected subset $\Gamma$ of $\R^3$ is called \textit{locally flat surface} if, for every point $x \in \Gamma$, there exist an open neighborhood $U_x$ of $x$ in $\R^3$ and a homeomorphism $\phi_x:U_x \lra \R^3$ such that $\phi_x(U_x \cap \Gamma)=P$, where $P$ is the coordinate plane $\{(x,y,z) \in\R^3 \, | \, z=0\}$. Suppose that $\Gamma$ is a locally flat surface. Thanks to the Jordan--Brouwer Separation Theorem, $\R^3 \setminus \Gamma$ consists of two connected components, one bounded $\Omega(\Gamma)$ and one unbounded $\Omega'(\Gamma)$, each of which has $\Gamma$ as its boundary (see \cite{Mun84}). In particular, $\Gamma$ is an orientable surface; topologically, a $2$-sphere with $g_\Gamma$ handles, where $g_\Gamma$ is called genus of $\Gamma$. There exist an open neighborhood $N$ of $\Gamma$ in $\overline{\Omega(S)}$ and a homeomorphism $\psi:\Gamma \times [0,1) \lra N$ such that $\psi(x,0)=x$ for every $x \in \Gamma$ (see \cite{Brown62}). The neighborhood $N$ is called \emph{collar} of $\Gamma$ in $\overline{\Omega(S)}$. The surface $\Gamma$ has a similar collar in $\overline{\Omega'(\Gamma)}$.

Let $\Omega$ be a bounded domain of $\R^3$. The boundary $\partial\Omega$ of $\Omega$ is said to be \textit{locally flat} if it is a finite union of pairwise disjoint locally flat surfaces. Suppose that $\Omega$ has locally flat boundary. Denote by $\Gamma_0,\Gamma_1,\ldots,\Gamma_p$ the connected components of $\partial\Omega$, which are locally flat surfaces. Without loss of generality, we can assume that $\Gamma_0$ is the ``external'' connected component of $\partial\Omega$ and $\Gamma_1,\ldots,\Gamma_p$ are the ``internal'' ones; namely, $\Omega=\Omega(\Gamma_0) \cap \bigcap_{i=1}^p\Omega'(\Gamma_i)$. Since each $\Gamma_i$ has a collar both in $\overline{\Omega(\Gamma_i)}$ and in $\overline{\Omega'(\Gamma_i)}$, it follows that $\partial\Omega$ has a collar both in $\overline{\Omega}$ and in $\R^3 \setminus \Omega$ too. Suppose that $\Omega$ is also polyhedral; namely, its closure $\overline{\Omega}$ in $\R^3$ is also triangulable. Let $\T$ be a (tetrahedral) triangulation of $\overline{\Omega}$ and let $\T_\partial$ be the triangulation induced by $\T$ on $\partial\Omega$. The reader observes that each edge in $\T_\partial$ belongs to exactly two faces of $\T_\partial$ and each face in $\T_\partial$ belongs to a unique tetrahedron of $\T$.

It is worth recalling that there is no topological difference between locally flat, polyhedral and smooth domains, where ``smooth'' means ``of class $\mc{C}^\infty$''. In fact, given any bounded domain $\Omega$ with locally flat boundary, there exist homeomorphisms $h,k:\R^3 \lra \R^3$ such that the domain $h(\Omega)$ has polyhedral closure and the domain $k(\Omega)$ has smooth boundary.
For further information on the topology of  three-dimensional domains, we refer the reader to \cite{BFG10}.

\textit{Throughout the remainder of this paper, $\Omega$ will denote a bounded polyhedral domain of $\R^3$ with locally flat boundary}.


\section{Preliminary homological notions} \label{sec:preliminaries}

This section is organized in three subsections. In the first one, we recall some basic concepts of simplicial homology theory concerning the fixed bounded polyhedral domain $\Omega$ of $\R^3$ with locally flat boundary, equipped with a triangulation $\T$. The second subse\-ction deals with the description of part of the dual complex of $\T$ and the related definitions of complete dual graph, coil and plug of $\T$. In the last subse\-ction, we recall the notion and some properties of linking number.


\subsection{Cycles, boundaries and homological Seifert surfaces} \label{subsec:cbs}

We start by recalling some notions of homology theory.
The basic concept is that of chain. A 0-chain of $\R^3$ is a finite formal linear combination $\sum_{i=1}^n p_i {\bf v}_i$ of points ${\bf v}_i \in \R^3$ with integer coefficients $p_i$. We denote by $C_0(\R^3, \Z)$ the abelian group of 0-chains of $\R^3$.

Given two different points ${\bf a},{\bf b}$ in $\R^3$, we denote by $[{\bf a}, {\bf b}]$ the oriented segment of $\R^3$ from ${\bf a}$ to ${\bf b}$; namely, the segment $\{t\mb{a}+s\mb{b} \in \R^3 \, | \, t,s \geq 0, t+s=1\}$ of $\R^3$ of vertices $\mb{a},\mb{b}$, together with the ordering $(\mb{a},\mb{b})$ of its vertices. The segment of $\R^3$ of vertices $\mb{a}$, $\mb{b}$ is called support of $[\mb{a},\mb{b}]$ and it is denoted by $|[\mb{a},\mb{b}]|$. The unit tangent vector $\bs{\tau}([\mb{a},\mb{b}])$ of the oriented segment $[{\bf a}, {\bf b}]$ is given by  $\bs{\tau}([{\bf a}, {\bf b}]):=\frac{{\bf b} - {\bf a}}{| {\bf b} - {\bf a} |}$. A (piecewise linear) $1$-chain of $\R^3$ is a finite formal linear combination $\sum_{i=1}^m a_i e_i$ of oriented segments $e_i=[{\bf a}_i, {\bf b}_i]$ of $\R^3$ with integer coefficients $a_i$. We identify $[{\bf b}, {\bf a}]= -[{\bf a}, {\bf b}]$ and we denote by $C_1(\R^3, \Z)$ the abelian group of $1$-chains in $\R^3$.

Analogously, if ${\bf a}$, ${\bf b}$, ${\bf c}$ are three different not aligned points in $\R^3$, we denote by $[{\bf a}, {\bf b}, {\bf c}]$ the oriented triangle of $\R^3$; namely, the triangle $\{t\mb{a}+s\mb{b}+u\mb{c} \in \R^3 \, | \, t,s,u \geq 0, t+s+u=1\}$ of $\R^3$ of vertices $\mb{a}$, $\mb{b}$, $\mb{c}$, together with the ordering $(\mb{a},\mb{b},\mb{c})$ of its vertices. The triangle of $\R^3$ of vertices $\mb{a},\mb{b},\mb{c}$
is called support of $[\mb{a},\mb{b},\mb{c}]$ and it is denoted by  $|[\mb{a},\mb{b},\mb{c}]|$. The unit normal vector $\bs{\nu}([\mb{a},\mb{b},\mb{c}])$ of the oriented triangle $[{\bf a}, {\bf b}, {\bf c}]$ is obtained by the right hand rule: $\bs{\nu}([{\bf a}, {\bf b}, {\bf c}]):=\frac{ ({\bf b} - {\bf a}) \times ({\bf c} - {\bf a})}{| ({\bf b} - {\bf a}) \times ({\bf c} - {\bf a})|}$. A (piecewise linear) 2-chain of $\R^3$ is a finite formal linear combination $\sum_{i=1}^p b_i f_i$ of oriented triangles $f_i=[{\bf a}_i, {\bf b}_i, {\bf c}_i]$ of $\R^3$ with integer coefficients $b_i$. If $\rho: \{{\bf a}, {\bf b}, {\bf c} \} \lra \{ {\bf a}, {\bf b}, {\bf c}\}$ is a permutation, we identify $[\rho({\bf a}), \rho({\bf b}), \rho({\bf c})]=[{\bf a}, {\bf b}, {\bf c}]$ if $\boldsymbol{\nu}([\rho({\bf a}), \rho({\bf b}), \rho({\bf c})])=\boldsymbol{\nu}([{\bf a}, {\bf b}, {\bf c}])$ and $[\rho({\bf a}), \rho({\bf b}), \rho({\bf c})]=-[{\bf a}, {\bf b}, {\bf c}]$ if $\boldsymbol{\nu}([\rho({\bf a}), \rho({\bf b}), \rho({\bf c})])=-\boldsymbol{\nu}([{\bf a}, {\bf b}, {\bf c}])$. We denote by $C_2(\R^3, \Z)$ the abelian group of 2-chains in $\R^3$.

Finally, if ${\bf a}$, ${\bf b}$, ${\bf c}$, ${\bf d}$ are four different not coplanar points in $\R^3$, we denote by $[{\bf a}, {\bf b}, {\bf c}, {\bf d}]$ the oriented tetrahedron of $\R^3$; namely, the tetrahedron $\{t\mb{a}+s\mb{b}+u\mb{c}+v\mb{d} \in \R^3 \, | \, t,s,u,v \geq 0, t+s+u+v=1\}$ of $\R^3$ of vertices $\mb{a}$, $\mb{b}$, $\mb{c}$, $\mb{d}$, together with the ordering $(\mb{a},\mb{b},\mb{c},\mb{d})$ of its vertices. The tetrahedron of $\R^3$ of vertices $\mb{a},\mb{b},\mb{c},\mb{d}$ is called support of the oriented tetrahedron $[\mb{a},\mb{b},\mb{c},\mb{d}]$ and it is denoted by $|[\mb{a},\mb{b},\mb{c},\mb{d}]|$.  A (piecewise linear) 3-chain of $\R^3$ is a finite formal linear combination $\sum_{i=1}^q d_i t_i$ of oriented tetrahedra $t_i=[{\bf a}_i, {\bf b}_i, {\bf c}_i, {\bf d}_i]$ of $\R^3$ with integer coefficients $d_i$. If $\rho: \{{\bf a}, {\bf b}, {\bf c}, {\bf d} \} \lra \{ {\bf a}, {\bf b}, {\bf c}, {\bf d}\}$ is a permutation, we identify $[\rho({\bf a}), \rho({\bf b}), \rho({\bf c}), \rho({\bf d})]=[{\bf a}, {\bf b}, {\bf c}, {\bf d}]$ if $\rho $ is an even permutation and $[\rho({\bf a}), \rho({\bf b}), \rho({\bf c}), \rho({\bf d})]=-[{\bf a}, {\bf b}, {\bf c}, {\bf d}]$ if $\rho $ is an odd permutation. We denote by $C_3(\R^3, \Z)$ the abelian group of 3-chains in $\R^3$.

We remark that, if all the coefficients in one of the preceding finite formal linear combinations are equal to zero, then we obtain the null element of the corresponding abelian group.

Let $k \in \{0,1,2,3\}$ and let $c=\sum_{i=1}^rc_iz_i$ be a $k$-chain of $\R^3$, where the $c_i$'s are integers and the $z_i$'s are points, oriented segments, oriented triangles or oriented tetrahedra of $\R^3$ if $k=0,1,2$ or $3$, respectively. Denote by $I_c$ the set of indices $i \in \{1,\ldots,r\}$ such that $c_i \neq 0$. The support $|c|$ of $c$ is the subset of $\R^3$ defined as the union $\bigcup_{i \in I_c}|z_i|$. 
We precise that $|c|=\emptyset$ if $c=0$. Moreover $|z_i|=\{z_i\}$ (and hence $|c|=\{z_i \in \R^3 \, | \, c_i \neq 0\}$) if $k=0$.

For every $k \in \{1,2,3\}$, let us define the boundary operator $\partial_k: C_k(\R^3;\Z) \lra C_{k-1}(\R^3;\Z)$. For every oriented segment $e=[{\bf a},{\bf b}]$, for every oriented triangle $f=[{\bf a}, {\bf b}, {\bf c}]$, and for every oriented tetrahedron $t=[{\bf a}, {\bf b}, {\bf c}, {\bf d}]$ of $\R^3$, we set $\partial_1 e:={\bf b}-{\bf a}$, $\partial_2f:=[{\bf b},{\bf c}] - [{\bf a},{\bf c}] + [{\bf a},{\bf b}]$ and $\partial_3t:=
[{\bf b},{\bf c},{\bf d}] -[{\bf a},{\bf c},{\bf d}]+[{\bf a},{\bf b},{\bf d}] -
[{\bf a},{\bf b},{\bf c}]$. Now we extend these definitions to all the $k$-chains of $\R^3$ by linearity. The reader observes that $\partial_1(\partial_2f)=(\mb{b}-\mb{a})+(\mb{c}-\mb{b})-(\mb{c}-\mb{a})=0$. In this way, by linearity, we have that $\partial_1 \circ \partial_2=0$ on the whole $C_2(\R^3;\Z)$. Analogously, we have that $\partial_2 \circ \partial_3=0$ on the whole $C_3(\R^3;\Z)$.

A $1$-chain $\gamma$ of $\R^3$ is called \textit{$1$-cycle of $\R^3$} if $\partial_1 \gamma=0$. The $1$-chain $\gamma$ is said to be a \textit{$1$-boundary of $\R^3$} if there exists a $2$-chain $S$ of $\R^3$ such that $\partial_2S=\gamma$. In this situation, we say that $S$ is a \textit{homological Seifert surface of $\gamma$ in $\R^3$}. Since $\partial_1 \circ \partial_2=0$, every $1$-boundary of $\R^3$ is also a $1$-cycle of $\R^3$. Actually, $\R^3$ is contractible (namely, it can be continuously deformed to a point) and hence the converse is true as well: every $1$-cycle of $\R^3$ is also a $1$-boundary of $\R^3$. In other words, a $1$-chain  of $\R^3$ has a homological Seifert surface in $\R^3$ if and only if it is a $1$-cycle of $\R^3$.

Let $Y$ be a subset of $\R^3$ and let $\eta$ be a $1$-cycle of $\R^3$ with $|\eta| \subset Y$. We say that \textit{$\eta$ bounds in $Y$} if $\eta$ admits a homological Seifert surface $S$ in $\R^3$ with $|S| \subset Y$. Given another $1$-cycle $\eta'$ of $\R^3$ with $|\eta'| \subset Y$, we say that \textit{$\eta$ and $\eta'$ are homologous in $Y$} if $\eta-\eta'$ bounds in $Y$.

Let $\Omega$ be the fixed bounded polyhedral domain of $\R^3$ with locally flat boundary and let $\T=(V,E,F,K)$ be a finite triangulation of $\overline\Omega$, where $V$ is the set of vertices, $E$ the set of edges, $F$ the set of faces and $K$ the set of tetrahedra of $\T$.

Let us fix an orientation (namely, an ordering of vertices) of each edge, face and tetrahedron of $\T$. This can be done as follows. Choose a total ordering $({\bf v}_1,\ldots,{\bf v}_{\mk{v}})$ of the elements of $V$. If $e=\{\mb{v}_i,\mb{v}_j\} \in E$ is an edge of $\T$ of vertices $\vv_i,\vv_j$ with $1 \leq i < j \leq \mk{v}$, then $e$ determines the oriented segment $[\vv_i,\vv_j]$ of $\R^3$. Analogously, the face $f=\{\vv_i,\vv_j,\vv_k\} \in F$ of $\T$ of vertices $\vv_i,\vv_j,\vv_k$ with $1 \leq i<j<k \leq \mk{v}$ and the tetrahedron $t=\{\vv_i,\vv_j,\vv_k, \vv_l\} \in K$ of $\T$ with $1 \leq i<j<k<l \leq \mk{v}$ determine the oriented triangle $[\vv_i,\vv_j,\vv_k]$ of $\R^3$ and the oriented tetrahedron $[\vv_i,\vv_j,\vv_k, \vv_l]$ of $\R^3$, respectively. In what follows, we denote again by $e$, $f$ and $t$, the \emph{oriented edges} of $\T$, the \emph{oriented faces} of $\T$ and the \emph{oriented tetrahedra} of $\T$, respectively. We indicate by $\E$, $\F$ and $\K$ the sets of oriented edges, oriented faces and oriented tetrahedra of $\T$, respectively.

A $k$-chain of $\T$ is a formal linear combination of vertices in $V$, oriented edges in $\E$, oriented faces in $\F$ and oriented tetrahedra in $\K$ for $k=0,1,2$ and $3$, respectively. We denote by $C_k(\T;\Z)$ the abelian subgroup of $C_k(\R^3;\Z)$ consisting of all $k$-chains of $\T$. Observe that the boundary operators $\partial_k$ preserve the chains of $\T$; namely, $\partial_k(C_k(\T;\Z)) \subset C_{k-1}(\T;\Z)$ if $k \in \{1,2,3\}$.

A $1$-chain $\gamma$ of $\T$ is called \textit{$1$-cycle of $\T$} if $\partial_1 \gamma = 0$, and it is called \textit{$1$-boundary of $\T$} if there exists a $2$-chain $S$ of $\T$ such that $\partial_2 S=\gamma$. Two $1$-cycles $\gamma$ and $\gamma'$ of $\T$ are said to be \textit{homologous in $\T$} if $\gamma-\gamma'$ is a $1$-boundary of $\T$. Denote by $Z_1(\T;\Z)$ the set of all $1$-cycles of $\T$ and by $B_1(\T;\Z)$ the set of all $1$-boundaries of $\T$. Since $\partial_1$ and $\partial_2$ are linear maps, and $\partial_1 \circ \partial_2=0$, we have that $Z_1(\T;\Z)$ and $B_1(\T;\Z)$ are abelian subgroups of $C_1(\T;\Z)$, and $B_1(\T;\Z) \subset Z_1(\T;\Z)$.

These concepts allow to define the first homology group $H_1(\T;\Z)$ of $\T$ as the abelian group of all homology classes of
$1$-cycles of $\T$. More precisely, we have:
\[
H_1(\T;\Z):=Z_1(\T;\Z)/B_1(\T;\Z).
\]
This quotient group is a free abelian group; namely, it is isomorphic to $\Z^g$, where $g$ is the rank of $H_1(\T;\Z)$. The integer $g$ does not depend on $\T$, but only on $\overline{\Omega}$, and is called first Betti number of $\overline{\Omega}$ (see Munkres~\cite[p.\ 24]{Mun84}). For this reason, one can write $H_1(\overline{\Omega};\Z)$ in place of $H_1(\T;\Z)$. The group $H_1(\overline{\Omega};\Z)$ contains many geometric and analytic informations concerning $\overline{\Omega}$. For example, thanks to the Hodge decomposition theorem, we know that $g$ is equal to the dimension of the real vector space of all harmonic vector fields of $\Omega$ tangent to the boundary $\partial\Omega$.

It is worth recalling that $\overline{\Omega}$ is homologically trivial (that is, $g=0$) if and only if it is simply connected (see \cite[Corollary 3.5]{BFG10} for a proof). This equivalence continues to hold for $2$-dimensional locally flat polyhedral domains, but it is false in dimension $\geq 4$ (see \cite[Remarks 3.9 and 3.10]{BFG10}).

Let $\T_\partial=(V_{\partial},E_{\partial},F_{\partial})$ be the triangulation of $\partial\Omega$ induced by $\T$; namely, we have that $V_\partial=V \cap \partial\Omega$, $E_\partial$ is the set of edges of $\T$ with vertices in $V_\partial$ and $F_\partial$ is the set of faces of $\T$ with vertices in $V_\partial$. Denote by $\E_\partial$ and $\F_\partial$ the sets of oriented edges and of oriented faces of $\T$ determined by the edges in $E_\partial$ and the faces in $F_\partial$, respectively. We have:
\begin{center}
$\E_\partial=\big\{e \in \E \, \big| \, |e| \subset \partial\Omega\big\} \;$ and $\; \F_\partial=\big\{f \in \F \, \big| \, |f| \subset \partial\Omega\big\}$.
\end{center}

A $1$-chain of $\T_\partial$ is a formal linear combination of oriented edges in $\E_\partial$ and a $2$-chain of $\T_\partial$ a formal linear combination of oriented faces in $\F_\partial$. We denote by $C_k(\T_\partial;\Z)$ the abelian subgroup of $C_k(\T;\Z)$ consisting of $k$-chains of $\T_\partial$ for $k=1,2$. The notions of $1$-cycle and of $1$-boundary of $\T_\partial$ can be defined in the natural way: a $1$-chain $\gamma$ of $\T_\partial$ is a $1$-cycle of $\T_\partial$ if $\partial_1 \gamma = 0$, and it is a $1$-boundary of $\T_\partial$ if there exists a $2$-chain $S$ of $\T_\partial$ such that $\partial_2 S=\gamma$. The first homology group $H_1(\T_\partial;\Z)$ of $\T_\partial$ is the quotient group $\ker(\partial_1)$ modulo $\mr{Image}(\partial_2)$:
\[
H_1(\T_\partial;\Z):=\ker(\partial_1)/\mr{Image}(\partial_2).
\]
The isomorphic class of the group $H_1(\T_\partial,\Z)$ does not depend on $\T_\partial$, but only on $\partial \Omega$. In this way, one can write $H_1(\partial\Omega;\Z)$ in place of $H_1(\T_\partial;\Z)$. The group $H_1(\partial\Omega;\Z)$ is free and its rank is equal to $2g$, where $g$ is the first Betti number of $\overline \Omega$ (see \cite[Section 3.4]{BFG10}).

Let us introduce the notions of corner edge, of corner face and of corner tetrahedron of $\T$.
Let $e=\{\mb{v},\mb{w}\}$ be an edge of $\T$. We say that $e$ is a \textit{corner edge of $\T$} if $e \in E_\partial$ and there exist two distinct vertices $\mb{z}^*$ and $\mb{z}^{**}$ in $V_\partial \setminus \{\mb{v},\mb{w}\}$ such that the $3$-sets $f^*=\{\mb{v},\mb{w},\mb{z}^*\}$ and $f^{**}=\{\mb{v},\mb{w},\mb{z}^{**}\}$ are faces of $\T$ in $F_\partial$, and the $4$-set $t^*=\{\mb{v},\mb{w},\mb{z}^*,\mb{z}^{**}\}$ is a tetrahedron in $\T$. If $e$ has this property, then we call $f^*$ and $f^{**}$ \textit{corner faces of $\T$ associated with $e$}, and $t^*$ \textit{corner tetrahedron of $\T$ associated with $e$}, see Figure~\ref{fig:corner}. A corner face of $\T$ associated with some corner edge of $\T$ is called \textit{corner face of $\T$}. Similarly, a corner tetrahedron of $\T$ associated with some corner edge of $\T$ is called a \textit{corner tetrahedron of $\T$}.

\begin{figure}[!htb]
\centering
 \includegraphics[width=.55\textwidth]{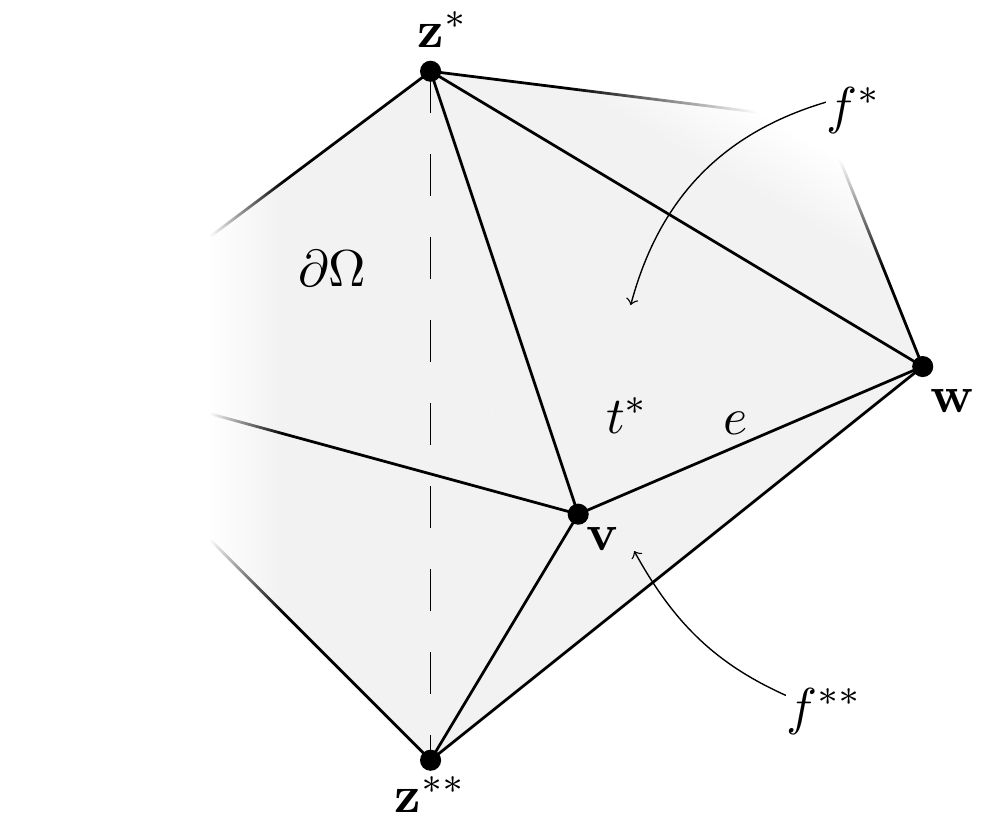}
  \caption{The corner edge $e$ and the corner faces $f^*$ and $f^{**}$.} \label{fig:corner}
\end{figure}

We denote by $E_\partial^\sangle$, $F_\partial^\sangle$ and $K_\partial^\sangle$ the sets of corner edges, of corner faces and of corner tetrahedra of $\T$, respectively. Moreover, we indicate by $\E_\partial^\sangle$ the sets of oriented edges in $\E_\partial$ determined by the corner edges of $\T$. Given a $1$-chain $\gamma=\sum_{e \in \E}a_ee$ of $\T$, we say that $\gamma$ is \textit{corner-free} if it does not contain any corner oriented edge; namely, if $a_e=0$ for every $e \in \E_\partial^\sangle$. Moreover, we call $\gamma$ \textit{internal} if it does not contain any boundary oriented edge; namely, if $a_e=0$ for every $e \in \E_\partial$. Evidently, if $\gamma$ is internal, then it is also corner-free. Similarly, given a $2$-chain $S=\sum_{f \in \F}b_ff$ of $\T$, we say that $S$ is \textit{internal} if it does not contain any boundary oriented face; namely, if $b_f=0$ for every $f \in \F_\partial$. The reader observes that, if $\T$ is the first barycentric subdivision of some triangulation of $\overline{\Omega}$, then $E_\partial^\sangle=\emptyset$ and hence every $1$-chain of $\T$ is corner-free. On the other hand, there are examples in which $E_\partial^\sangle \neq \emptyset$: if $\overline{\Omega}$ is a tetrahedron of $\R^3$ equipped with its natural triangulation $\T$, then $E_\partial^\sangle=E_\partial \neq \emptyset$.

We conclude this subsection by introducing the notions of homological Seifert surface and of internal homological Seifert surface. 

\begin{defin}
Given a $1$-boundary $\gamma$ of $\T$, we say that a $2$-chain $S$ of $\T$ is a \emph{homological Seifert surface of $\gamma$ in $\T$} if $\partial_2 S=\gamma$. If, in addition, $S$ is internal, then we call $S$ \emph{internal homological Seifert surface of $\gamma$ in $\T$}.
\end{defin}


\subsection{Complete dual graph, coils and plugs} \label{subsec:cdg-coils}

We begin by describing part of the closed block dual  barycentric complex of $\T$ (see \cite[Section 64]{Mun84} for the general definition).

Denote by $B:V \cup E \cup F \cup K \lra \R^3$ the barycenter map: if $\vv \in V$, $\ell=\{\vv,\mb{w}\} \in E$, $g=\{\vv,\mb{w},\mb{y}\} \in F$ and $t=\{\vv,\mb{w},\mb{y},\mb{z}\} \in K$, then we have $B(\vv)=\vv$, $B(\ell)=(\vv+\mb{w})/2$, $B(g)=(\vv+\mb{w}+\mb{y})/3$ and $B(t)=(\vv+\mb{w}+\mb{y}+\mb{z})/4$.
Extend $B$ to the oriented edges in $\E$ and to the oriented faces in $\F$ in the natural way: if $e=[\vv,\mb{w}] \in \E$ and $f=[\vv,\mb{w},\mb{y}] \in \F$, then we set $B(e):=(\vv+\mb{w})/2$ and $B(f):=(\vv+\mb{w}+\mb{y})/3$.

Let us recall the definitions of dual vertices, of dual edges and of dual faces of $\T$. We equip the dual edges and the dual faces with the natural orientation induced by the right hand rule.
\begin{itemize}
 \item For every tetrahedron $t \in K$, the dual vertex $D(t)$ of $\T$ associated with $t$ is defined as the barycenter of $t$:
 \[
 D(t):=B(t).
 \]
We denote by $V'$ the set $\{D(t) \in \R^3 \, | \, t \in K\}$ of all dual vertices of $\T$.
 \item For every oriented face $f=[\vv,\mb{w},\mb{y}] \in \F$, the oriented dual edge $D(f)$ of $\T$ associated with $f$ is the element of $C_1(\R^3;\Z)$ defined as follows: if $K(f)$ denotes the set $\big\{t \in K \, \big| \, \{\vv,\mb{w},\mb{y}\} \subset t\big\}$; namely, the set of tetrahedra of $\T$ incident on $f$, we set
\[
 D(f):=\sum_{t \in K(f)} \sign \big(\bs{\nu}(f) \cdot \bs{\tau}([B(f),B(t)]) \big) \, [B(f),B(t)],
\]
where $\sign:\R \setminus \{0\} \lra \{-1,1\}$ denotes the function given by $\sign(s):=-1$ if $s<0$ and $\sign(s):=1$ otherwise.

$D(f)$ can be described as follows. If the (oriented) face $f$ is internal, then $f$ is the common face of two tetrahedra $t_1$ and $t_2$ of $\T$, and the support of $D(f)$ is the union of the segment joining $B(f)$ with $B(t_1)$ and of the segment joining $B(f)$ and $B(t_2)$, see Figure~\ref{fig:dual_edge} (on the left). If $f$ is a boundary face, then $f$ is face of just one tetrahedron $t$, and the support of $D(f)$ is the segment joining $B(f)$ with $B(t)$, see Figure~\ref{fig:dual_edge} (on the right). In both cases, $D(f)$ is endowed with the orientation induced by $f$ via the right hand rule.

\begin{figure}[!htb]
\centering
  \includegraphics[width=.85\textwidth]{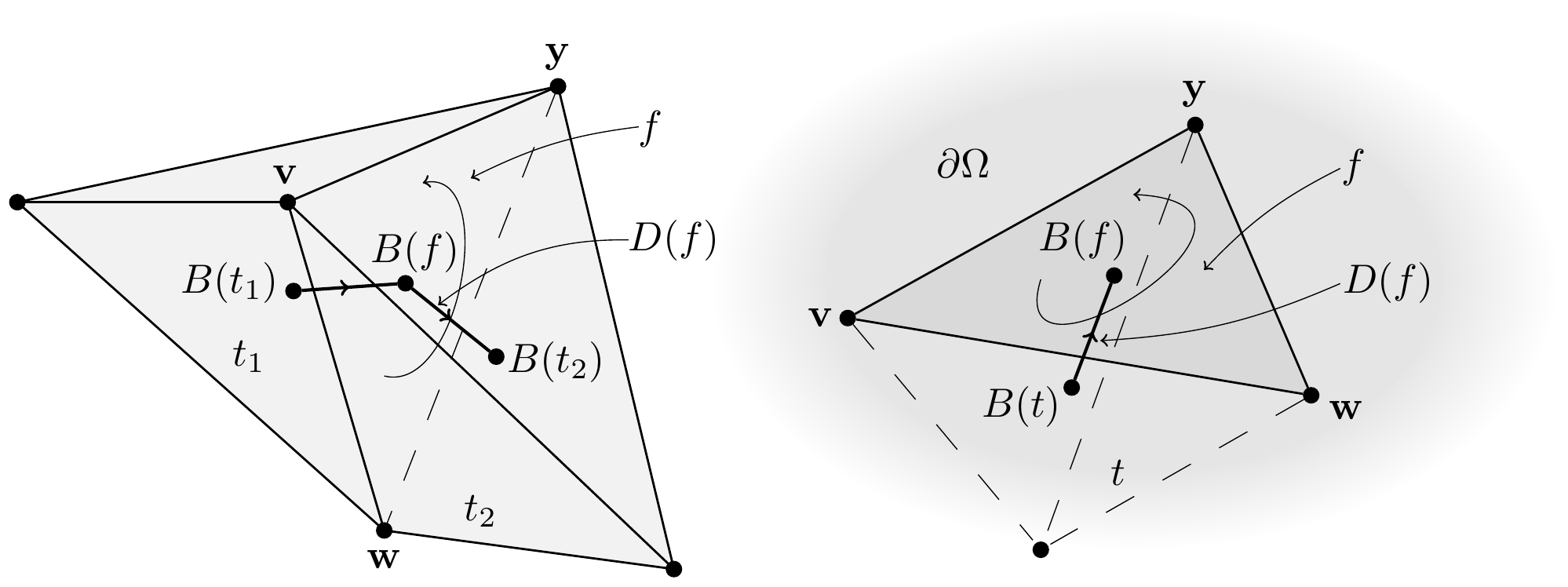}
  \caption{The dual edge $D(f)$ in the case of an internal face (on the left) and in the case of a boundary face (on the right).} \label{fig:dual_edge}
\end{figure}

We denote by $\E'$ the set $\{D(f) \in C_1(\R^3;\Z) \, | \, f \in \F\}$ of all oriented dual edges of $\T$. Moreover, we call (non-oriented) dual edge of $\T$ a $2$-subset $\{v',w'\}$ of $\R^3$ such that $\{v',w'\}=|\partial_1e'|$ for some $e' \in \E'$. We indicate by $E'$ the set of all (non-oriented) dual edges of $\T$.
 \item For every oriented edge $e=[\vv,\mb{w}] \in \E$, the oriented dual face $D(e)$ of $\T$ associated with $e$ is the element of $C_2(\R^3;\Z)$ defined as follows: if $F(e)$ denotes the set $\big\{f \in F \, \big| \, \{\vv,\mb{w}\} \subset f\big\}$; namely, the set of oriented faces of $\T$ incident on $e$, then we set
\[
D(e):=\sum_{f \in F(e)} \sum_{t \in K(f)} \sign \big( \bs{\tau}(e) \cdot \bs{\nu}([B(e),B(f),B(t)]) \big) \,   [B(e),B(f),B(t)],
\]
see Figure~\ref{fig:dual_face}. The reader observes that the support of $D(e)$ is the union of triangles of $\R^3$ obtained as the convex hull of the sets $\{B(e)\} \cup |D(f)|$, where $f$ varies in $F(e)$. Such triangles are oriented by $e$ via the right hand rule.

\begin{figure}[!htb]
\centering
  \includegraphics[width=.85\textwidth]{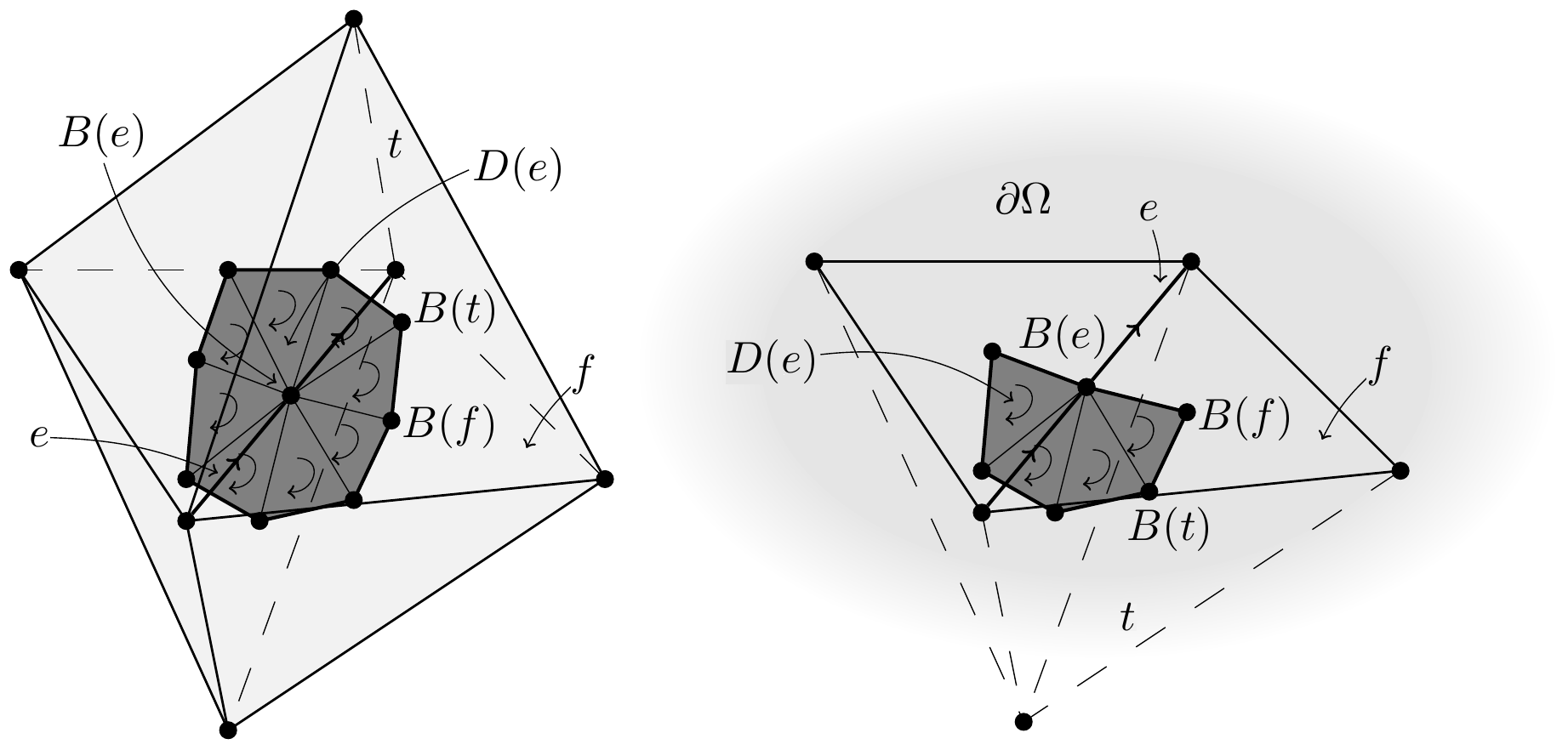}
  \caption{The dual face $D(e)$ in the case of an internal edge (on the left) and in the case of a boundary edge (on the right).} \label{fig:dual_face}
\end{figure}

We denote by $\F'$ the set $\{D(e) \in C_2(\R^3;\Z) \, | \, e \in \E\}$ of all oriented dual faces of $\T$.
\end{itemize}

The preceding three definitions determine the bijection $D:K \cup \F \cup \E \lra V' \cup \E' \cup \F'$ such that $D(K)=V'$, $D(\F)=\E'$ and $D(\E)=\F'$.

We need also to describe part of the closed block dual barycentric complex of the triangulation $\T_\partial$ of $\partial\Omega$ induced by $\T$. Recall that  $V_\partial$, $\E_\partial$ and $\F_\partial$ denote the sets of vertices, of oriented edges and of oriented faces of $\T_\partial$, respectively.

Let us define the dual vertices and the oriented dual edges of $\T_\partial$.

\begin{itemize}
 \item For every oriented face $f \in \F_\partial$, the dual vertex $D_\partial(f)$ of $\T_\partial$ associated with $f$ is defined as the barycenter of $f$:
 \[
 D_\partial(f):=B(f).
 \]
We denote by $V'_\partial$ the set $\{D_\partial(f) \in \R^3 \, | \, f \in \F_\partial\}$ of all dual vertices of $\T_\partial$.
 \item For every oriented edge $e \in \E_\partial$, the oriented dual edge $D_\partial(e)$ of $\T_\partial$ associa\-ted with $e$ is the element of $C_1(\R^3;\Z)$ defined as follows. Let $f_1$ and $f_2$ be the oriented faces in $\F_\partial$ incident on $e$, and let $\mb{n}(f_1)$ and $\mb{n}(f_2)$ be the outward unit normals of $\partial\Omega$ at $B(f_1)$ and at $B(f_2)$, respectively. Then we set
\[
D_\partial(e):=\sum_{i=1}^2 \sign\big(\bs{\tau}(e) \cdot (\mb{n}(f_i) \times \bs{\tau}([B(e),B(f_i)])) \big)[B(e),B(f_i)].
\]
$D_\partial(e)$ can be described as follows. By interchanging $f_1$ with $f_2$ if necessary, we can suppose that $f_1$ is on the left of $e$ and $f_2$ on the right of $e$ with respect to the orientation of $\partial\Omega$ induced by its outward unit vector field. Then we have:
\[
D_\partial(e)=[B(f_1),B(e)]+[B(e),B(f_2)],
\]
see Figure~\ref{fig:boundary_dual_edge}.
\begin{figure}[!htb]
\centering
  \includegraphics[width=.65\textwidth]{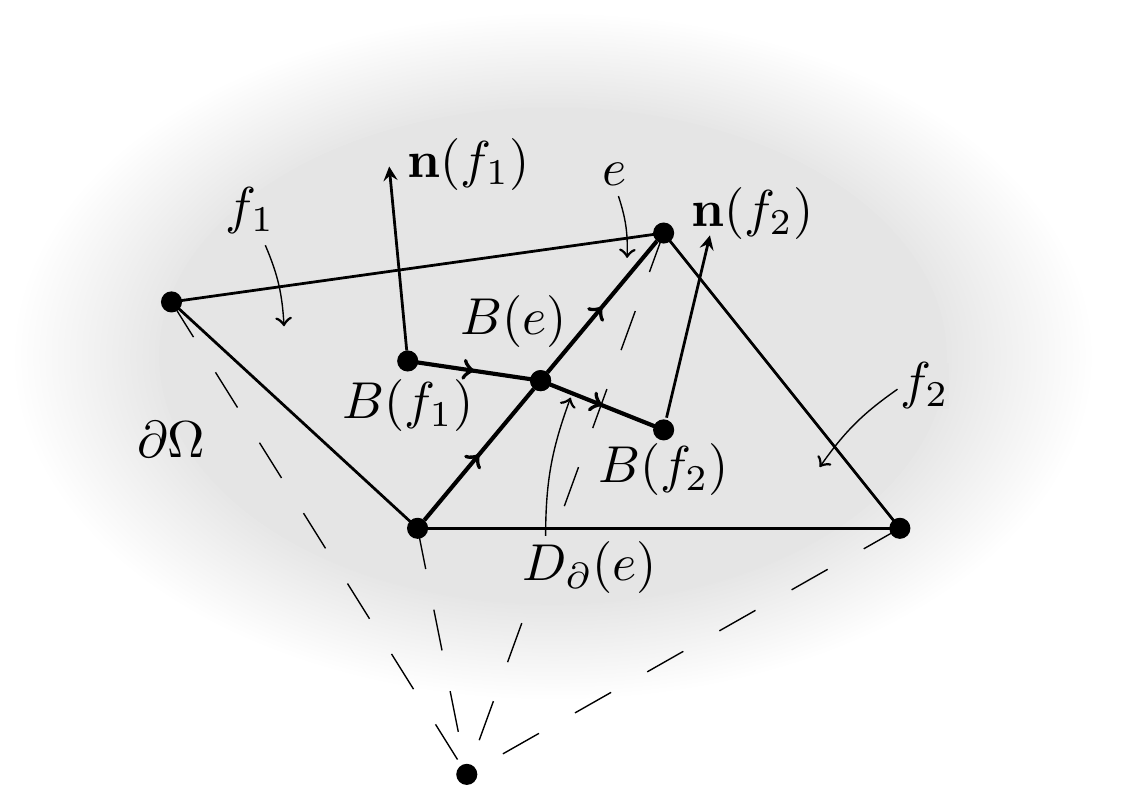}
  \caption{The boundary dual edge $D_\partial(e)$.} \label{fig:boundary_dual_edge}
\end{figure}

We denote by $\E'_\partial$ the set $\{D_\partial(e) \in C_1(\R^3;\Z) \, | \, e \in \E_\partial\}$; namely, the set of all oriented dual edges of $\T_\partial$. Moreover, we call (non-oriented) dual edge of $\T_\partial$ a $2$-subset $\{\mb{v}',\mb{w}'\}$ of $V'_\partial$ such that $\{\mb{v}',\mb{w}'\}=|\partial_1 e'|$ for some $e' \in \E'_\partial$. We indicate by $E'_\partial$ the set of all (non-oriented) dual edges of $\T_\partial$.
\end{itemize}

Let us give three definitions, which will prove to be useful later.

\begin{defin} \label{def:cdg}
We call $\A':=(V' \cup V'_\partial,E' \cup E'_\partial)$ \emph{complete dual graph of $\T$}.
A \emph{$1$-chain of $\A'$} is a formal linear combination of oriented dual edges in $\E' \cup \E'_\partial$ with integer coefficients. A $1$-chain $\gamma$ of $\A'$ is called \emph{$1$-cycle of $\A'$} if $\partial_1\gamma=0$. We denote by $C_1(\A';\Z)$ the abelian subgroup of $C_1(\R^3;\Z)$ consisting of all $1$-chains of $\A'$, and by  $Z_1(\A';\Z)$ the abelian subgroup of $Z_1(\R^3;\Z)$ consisting of all $1$-cycles of $\A'$.
\end{defin}

\begin{defin} \label{def:coil}
For every $e \in \E$, we define the \emph{coil of $e$ (in $\T$)}, denoted by $\coil(e)$, as the $1$-cycle of $\A'$ given by
\[
\coil(e):=\partial_2 D(e).
\]
\end{defin}

The reader observes that, for every $e \in \E_\partial$, $\coil(e)-D_\partial(e)$ is a $1$-chain of $\A'$, whose expression as a formal linear combination contains only oriented edges in $\E'$; namely, $\coil(e)-D_\partial(e)=\sum_{e' \in \E' \cup \E'_\partial}a_{e'}e'$ for some (unique) integer $a_{e'}$ such that $a_{e'}=0$ for every $e' \in \E'_\partial$.

Let us introduce the notion of plug of $\T$.

Given a dual edge $e' \in E'$, we say that $e'$ is a \emph{plug of $\T$} if there exists a face $f \in F_\partial$ such that $e'=\{B(f),B(t)\}$, where $t$ is the unique tetrahedron in $\T$ containing $f$. Such a plug $e'$ is said to be \emph{induced by $f$}. The plug $e'$ is called \emph{corner plug of $\T$} if it is induced by a corner face $f \in F_\partial^\sangle$, see Figure~\ref{fig:plug} (on the right). On the contrary, if the face inducing $e'$ belongs to $F_\partial \setminus F_\partial^\sangle$, then $e'$ is called \emph{regular plug of $\T$}, see Figure~\ref{fig:plug} (on the left).
Let $\mr{J}_{\T}$ be the set of all plugs of $\T$, and let $\mr{J}_{\T}^\sangle$ and $\mr{J}_{\T}^r$ be the subsets of $\mr{J}_{\T}$ consisting of corner plugs and of regular plugs of $\T$, respectively.

\begin{figure}[!htb]
\centering
  \includegraphics[width=\textwidth]{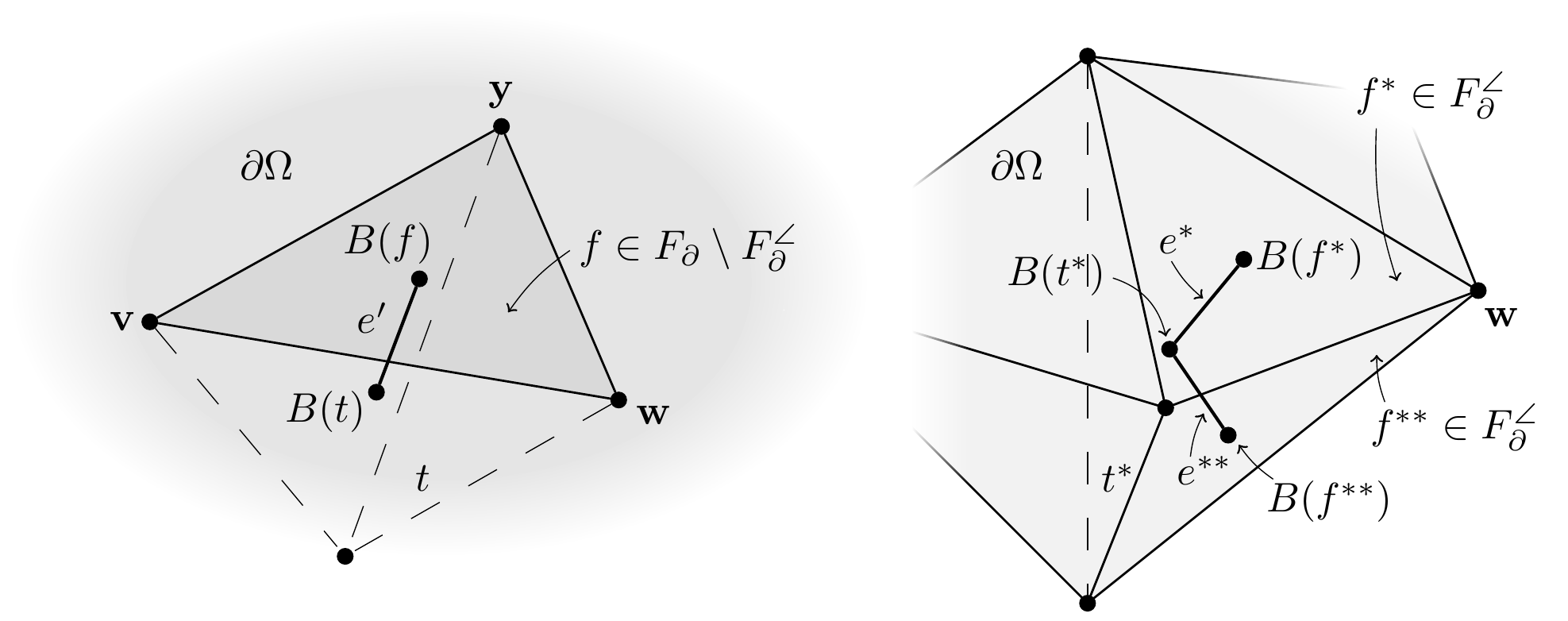}
  \caption{A regular plug $e'$ (on the left) and two corner plugs $e^*$ and $e^{**}$ (on the right).} \label{fig:plug}
\end{figure}

\begin{defin} \label{def:plug-set}
Given a subset $J$ of $\mr{J}_{\T}$, we say that $J$ is a \emph{plug-set of $\T$} if, for every $e',e'' \in J$ with $e' \neq e''$, $e'$ and $e''$ do not have any vertex in common; namely, $e' \cap e''=\emptyset$. Moreover, we say that such a plug-set $J$ is \emph{maximal} if it does not exist any plug-set of $\T$, which strictly contains $J$.
\end{defin}

\begin{remark} \label{maximal-plug-set}
Notice that a regular plug does not intersect any other plug so if $E_\partial^\sangle=\emptyset$ (or, equivalently, if $K_\partial^\sangle=\emptyset$), then all the plugs of $\T$ are regular and hence the set $\mr{J}_{\T}$ itself is the unique maximal plug-set of $\T$. Suppose $E_\partial^\sangle \neq \emptyset$. In this case, a subset $J$ of $\mr{J}_{\T}$ is a maximal plug-set of $\T$ if and only if it can be costructed as follows. For every $t \in K_\partial^\sangle$, choose one of the corner faces of $\T$ contained in $t$ and denote it by $f_t^\sangle$. Define $F^\sangle:=\{f_t^\sangle \in F_\partial^\sangle \, | \, t \in K_\partial^\sangle\}$ and indicate by $J'$ the set of corner plugs of $\T$ induced by the corner faces in $F^\sangle$. Then $J=\mr{J}_{\T}^r \cup J'$.
\end{remark}


\subsection{Linking number, recognition of 1-boundaries  and retractions} \label{subsec:lk-recogn-retract}

\noindent \textbf{Linking number.} We begin by recalling the notion of linking number. Consider two $1$-cycles $\gamma$ and $\eta$ of $\R^3$ with disjoint supports; namely, $|\gamma| \cap |\eta|=\emptyset$. A possible geometric way to define the linking number $\lk(\gamma,\eta)$ between $\gamma$ and $\eta$ is as follows.

Choose a homological Seifert surface $S_\eta=\sum_{q=1}^kb_qf_q$ of $\eta$ in $\R^3$. It is well-known (and easy to see) that there exists a $1$-cycle $\widehat{\gamma}=\sum_{p=1}^h\widehat{a}_p\widehat{e}_p$ homologous to $\gamma$ in $\R^3 \setminus |\eta|$ (and ``arbitrarily close to $\gamma$'' if necessary), which is transverse to $S_\eta$ in the following sense: for every $p \in \{1,\ldots,h\}$ and for every $q \in \{1,\ldots,k\}$, the intersection $|\widehat{e}_p| \cap |f_q|$ is either empty or consists of a single point, which does not belong to $|\partial_1 \widehat{e}_p| \cup |\partial_2 f_q|$.

For every $p \in \{1,\ldots,h\}$ and for every $q \in \{1,\ldots,k\}$, define $L_{pq}:=0$ if $|\widehat{e}_p| \cap |f_q|=\emptyset$ and $L_{pq}:=\sign(\bs{\tau}(\widehat{e}_p) \cdot \bs{\nu}(f_q))$ otherwise. The linking number $\lk(\gamma,\eta)$ between $\gamma$ and $\eta$ is the integer defined as follows:
\begin{equation} \label{eq:def-lk}
\lk(\gamma,\eta):=\sum_{p=1}^h\sum_{q=1}^k\widehat{a}_pb_qL_{pq}.
\end{equation}
This definition is well-posed: it depends only on $\gamma$ and $\eta$, not on the choice of $S_\eta$ and of $\widehat{\gamma}$. The reader observes that the preceding construction fully justifies the usual heuristic description of the linking number between $\gamma$ and $\eta$ as the number of times that $\gamma$ winds around~$\eta$.

The linking number has some remarkable properties. It is ``symmetric'' and ``bilinear'':
\[
\lk(\gamma,\eta)=\lk(\eta,\gamma),
\]
\[
\lk(a\gamma,\eta)=a \, \lk(\gamma,\eta) \; \text{ for every }a \in \Z
\]
and,  if $\gamma^* \in Z_1(\R^3;\Z)$ with $|\gamma^*| \cap |\eta|=\emptyset$,
\[
\lk(\gamma+\gamma^*,\eta)=\lk(\gamma,\eta)+\lk(\gamma^*,\eta) \, .
\]

The linking number is a homological invariant in the following sense: if a $1$-cycle $\gamma^*$ of $\R^3$ is homologous to $\gamma$ in $\R^3 \setminus |\eta|$, then
\begin{equation} \label{eq:homol-inv}
\lk(\gamma,\eta)=\lk(\gamma^*,\eta).
\end{equation}
In particular, we have:
\begin{equation} \label{eq:bounds}
\text{$\lk(\gamma,\eta)=0$ if $\gamma$ bounds in $\R^3 \setminus |\eta|$.}
\end{equation}

The linking number can be computed via an integral formula. Write $\gamma$ and $\eta$ explicitly:
$\gamma=\sum_{i=1}^na_ie_i$ and $\eta=\sum_{j=1}^mc_jg_j$ for some integer $a_i,c_j$ and for some oriented segment $e_i=[{\bf a}_i, {\bf b}_i]$ and $g_j=[\mb{c}_j,\mb{d}_j]$ of $\R^3$. The following Gauss formula holds:
\begin{equation} \label{eq:lk-gauss}
\lk(\gamma,\eta)=\frac{1}{4\pi}\sum_{i=1}^n\sum_{j=1}^m a_i c_j \left(\int_0^1\int_0^1 \frac{e_i(r)-g_j(s)}{|e_i(r)-g_j(s)|^3} \times \vec{e}_i\right) \cdot \vec{g}_j \, dr \, ds,
\end{equation}
where $\vec{e}_i:=\mb{b}_i-\mb{a}_i$, $\vec{g}_j:=\mb{d}_j-\mb{c}_j$ and  $e_i(r):=\mb{a}_i+r\vec{e}_i$, $g_j(s):=\mb{c}_j+s\vec{g}_j$ if $r,s \in [0,1]$. We refer the reader to \cite{BG12} for a fast algorithm to compute $\lk(\gamma,\eta)$ accurately, by means of an explicit expression of the preceding integral.

\vspace{.5em}

\noindent \textbf{Recognition of 1-boundaries.} The linking number can be used to recognize $1$-boundaries of $\T$ among $1$-cycles of $\T$. This is possible by the Alexander duality theorem. Indeed, such a theorem ensures that $H_1(\R^3 \setminus \overline{\Omega};\Z)$ is isomorphic to $H_1(\overline{\Omega};\Z)$, and hence to $\Z^g$ if $g$ is the first Betti number of $\overline{\Omega}$. Furthermore, if $\sigma^*_1,\ldots,\sigma^*_g$ are $1$-cycles of $\R^3$ with support in $\R^3 \setminus \overline{\Omega}$ whose homology classes in $\R^3 \setminus \overline{\Omega}$ form a basis of $H_1(\R^3 \setminus \overline{\Omega};\Z)$, then it holds:
\begin{equation*} 
\text{a $1$-cycle $\sigma$ of $\T$ is a $1$-boundary of $\T$ if and only if $\lk(\sigma,\sigma^*_i)=0$ for every $i \in \{1,\ldots,g\}$.}
\end{equation*}

\vspace{.5em}

\noindent \textbf{Retractions.} Now we define the ``retractions'' $R_+:Z_1(\T;\Z) \lra Z_1(\R^3;\Z)$ and $R_-:Z_1(\A';\Z) \lra Z_1(\R^3;\Z)$, and we prove an useful invariance property of certain linking numbers with respect to the application of such ``retractions''.

Let us define $R_+$. For every oriented edge $e=[\vv,\mb{w}]$ in $\E_\partial$, choose a tetrahedron $t_e \in K$ incident on $e$ (namely, $\{\vv,\mb{w}\} \subset t_e$), denote by $\mb{d}_e$ the barycenter of the triangle of $\R^3$ of vertices $\vv$, $\mb{w}$, $B(t_e)$, and define the $1$-chain $r_+(e)$ of $\R^3$ and the oriented triangle $S_e$ of $\R^3$ by setting
\[
r_+(e):=[\vv,\mb{d}_e]+[\mb{d}_e,\mb{w}]
\quad \mbox{and} \quad
S_e:=[\vv,\mb{d}_e,\mb{w}].
\]
The reader observes that $\partial_2S_e=r_+(e)-e$, see Figure~\ref{fig:rpiu}.
\begin{figure}[!htb]
\centering
 \includegraphics[width=.60\textwidth]{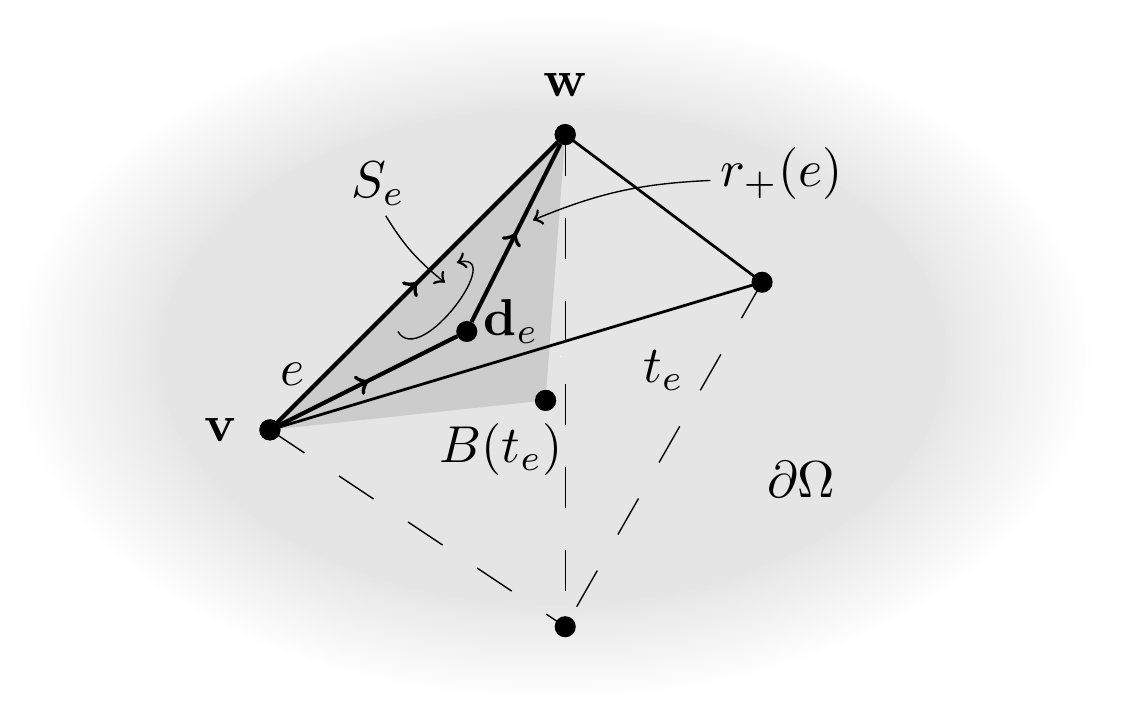}
  \caption{The $1$-chain $r_+(e)$ and the oriented triangle $S_e$.} \label{fig:rpiu}
\end{figure}

Given $\xi=\sum_{e \in \E}\alpha_ee \in Z_1(\T;\Z)$, we define:
\[
R_+(\xi):=\sum_{e \in \E \setminus \E_\partial} \alpha_e e + \sum_{e \in \E_\partial} \alpha_e r_+(e).
\]
Evidently, $R_+(\xi)$ belongs to $Z_1(\R^3;\Z)$ and $R_+(\xi)-\xi$ is a $1$-boundary of $\R^3$:
\begin{equation} \label{eq:R+}
\textstyle
R_+(\xi)-\xi=\partial_2\left(\sum_{e \in \E_\partial} \alpha_e S_e\right).
\end{equation}

Now we introduce $R_-$. First, we recall that, since $\partial\Omega$ is assumed to be locally flat, we know that it has a collar in $\R^3 \setminus \Omega$; namely, there exist an open neighborhood $U$ of $\partial\Omega$ in $\R^3 \setminus \Omega$ and a homeomorphism $\psi: \partial\Omega \times [0,1)  \lra U$, called collar of $\partial\Omega$ in $\R^3 \setminus \Omega$, such that $\psi(x,0)=x$ for every $x \in \partial\Omega$.

Let $e' \in \E'_\partial$. By definition of $\E'_\partial$, there exist, and are unique, $e \in \E_\partial$ and $f_1,f_2 \in \F_\partial$ such that $e'=D_\partial(e)=[B(f_1),B(e)]+[B(e),B(f_2)]$.
Thanks to the existence of a collar of $\partial\Omega$ in $\R^3 \setminus \Omega$, one can choose a point $\mb{x}_{e'} \in \R^3 \setminus \overline{\Omega}$ arbitrarily close to $B(e)$ with the following property: if $S'_{e'}$ is the $2$-chain of $\R^3$ defined by setting
\begin{equation} \label{eq:S'}
S'_{e'}:=[B(f_1),\mb{x}_{e'},B(e)]+[B(e),\mb{x}_{e'},B(f_2)],
\end{equation}
then $\overline{\Omega} \cap |S'_{e'}|=|e'|$. Denote by $r_-(e')$ the $1$-chain $[B(f_1),\mb{x}_{e'}]+[\mb{x}_{e'},B(f_2)]$ of $\R^3$, see Figure~\ref{fig:external}. Observe that $\partial_2S'_{e'}=r_-(e')-e'$.

\begin{figure}[!htb]
\centering
  \includegraphics[width=.65\textwidth]{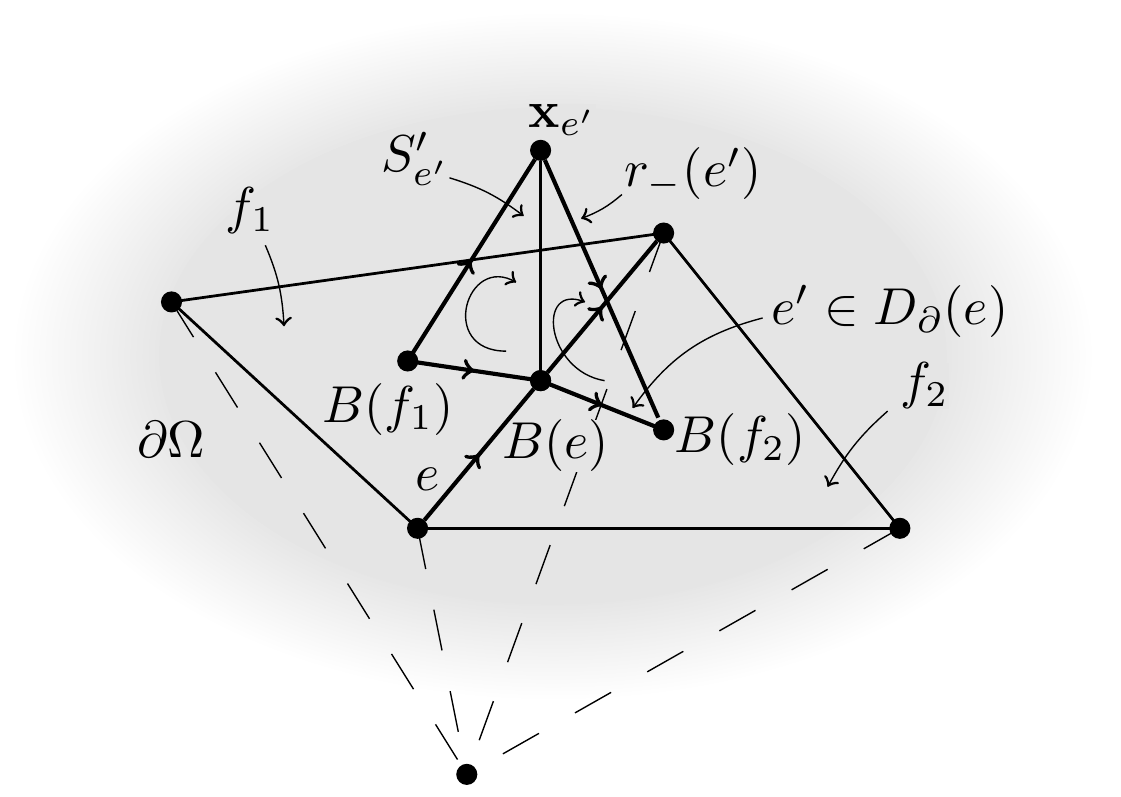}
  \caption{The $1$-chain $r_-(e')$ and the $2$-chain $S'_{e'}$.} \label{fig:external}
\end{figure}

For every $\xi'=\sum_{e' \in \E' \cup \E'_\partial} \alpha'_{e'}e' \in Z_1(\A';\Z)$, we define:
\begin{equation} \label{eq:R_-}
R_-(\xi'):=\sum_{e' \in \E'}\alpha'_{e'}e'+\sum_{e' \in \E'_\partial}\alpha'_{e'}r_-(e').
\end{equation}
We remark that $R_-(\xi')$ is a $1$-cycle of $\R^3$ and $R_-(\xi')-\xi'$ is a $1$-boundary of $\R^3$:
\begin{equation} \label{eq:R-}
\textstyle
R_-(\xi')-\xi'=\partial_2\left(\sum_{e' \in \E'_\partial}\alpha'_{e'}S'_{e'}\right).
\end{equation}

The following result holds true.

\begin{lemma} \label{lem:C+C-}
For every $\xi \in Z_1(\T;\Z)$ and for every $\xi' \in Z_1(\A';\Z)$, it holds:
\[
\lk\big(R_+(\xi),\xi'\big)=\lk\big(\xi,R_-(\xi')\big).
\]
\end{lemma}
\begin{proof}
First, observe that $|R_+(\xi)| \cap |\xi'|= \emptyset$, $|\xi| \cap |R_-(\xi')|=\emptyset$ and hence the linking numbers $\lk(R_+(\xi),\xi')$ and $\lk(\xi,R_-(\xi'))$ are defined. Moreover, it holds:
\begin{equation} \label{eq:+}
|R_+(\xi)| \cap \bigcup_{e' \in \E'_\partial}|S'_{e'}|=\emptyset
\end{equation}
and
\begin{equation} \label{eq:-}
|R_-(\xi')| \cap \bigcup_{e \in \E_\partial}|S_e|=\emptyset.
\end{equation}

By combining points \eqref{eq:R-} and \eqref{eq:-}, we obtain that $\xi'$ and $R_-(\xi')$ are homologous in $\R^3 \setminus |R_+(\xi)|$. Thanks to \eqref{eq:homol-inv}, we infer that $\lk(R_+(\xi),\xi')=\lk(R_+(\xi),R_-(\xi'))$. Similarly, points \eqref{eq:R+}, \eqref{eq:+} and \eqref{eq:homol-inv} ensure that $\lk(\xi,R_-(\xi'))=\lk(R_+(\xi),R_-(\xi'))$. It follows that $\lk(R_+(\xi),\xi')=\lk(\xi,R_-(\xi'))$, as desired.
\end{proof}

\begin{remark}
We have introduced the rectraction $R_-$ in order to simplify the proof of some results. However, it will be never used in the construction of the homological Seifert surfaces presented below.
\end{remark}


\section{The main results} \label{sec:results}


\subsection{The statements}

Consider the complete dual graph $\A'=(V' \cup V'_\partial, E' \cup E'_\partial)$ of $\T$.
Choose a spanning tree $\B'=(V' \cup V'_\partial,N')$ of $\A'$ and denote by $\mc{N}'$ the set of oriented dual edges in $\E' \cup \E'_{\partial}$ corresponding to $N'$; namely, we set
\[
\mc{N}':=\big\{e' \in \E' \cup \E'_{\partial} \, \big| \, |\partial_1e'| \in N'\big\}.
\]
We call $\mc{N}'$ \textit{set of oriented dual edges of $\B'$}.

Fix a dual vertex $\mb{a}' \in V' \cup V'_\partial$, we will consider as a root of $\B'$. Let us give the rigorous definition of ``(unique) $1$-chain $C'_{\mb{v}'}$ of $\B'$ from the root $\mb{a}'$ to another vertex $\mb{v}'$''. Consider a dual vertex $\mb{v}'$ in $V' \cup V'_\partial$. First, suppose $\mb{v}' \neq \mb{a}'$. Since $\B'$ is a tree, there exist, and are unique, a positive integer $m$ and an ordered sequence $(\mb{w}'_0,\mb{w}'_1,\ldots,\mb{w}'_m)$ of vertices in $V' \cup V'_\partial$ such that $\mb{w}'_0=\mb{a}'$, $\mb{w}'_m=\mb{v}'$, $\mb{w}'_i \neq \mb{w}'_j$ for every $i,j \in \{0,1,\ldots,m\}$ with $i \neq j$ and $\{\mb{w}'_{k-1},\mb{w}'_k\} \in N'$ for every $k \in \{1,\ldots,m\}$. In this way, for every $k \in \{1,\ldots,m\}$, there exist, and are unique, $e'_k \in \mc{N}'$ and $\delta_k \in \{-1,1\}$ such that $\partial_1(\delta_ke'_k)=\mb{w}'_k-\mb{w}'_{k-1}$. We can now define $C'_{\mb{v}'} \in C_1(\A';\Z)$ as follows:
\begin{equation} \label{eq:C'}
C'_{\mb{v}'}:=\sum_{k=1}^m\delta_ke'_k.
\end{equation}
Evidently, it holds: $\partial_1(C'_{\mb{v}'})=\mb{v}'-\mb{a}'$. If $\mb{v}'=\mb{a}'$, then we define $C'_{\mb{v}'}$ as the zero $1$-chain in $C_1(\A';\Z)$.

For every oriented dual edge $e' \in \E'\cup \E'_\partial$ with $\partial_1 e'=\mb{v}'-\mb{w}'$, we define the $1$-cycle $\cb(e')$ of $\A'$ by setting
\[
\cb(e'):=C'_{\mb{w}'}+e'-C'_{\mb{v}'}.
\]
The reader observes that $\cb(e')$ depends only on $\B'$ and on $e'$, and not on the choosen root $\mb{a}'$ of $\B'$. Moreover, if $e' \in \mc{N}'$, then $\cb(e')=0$.

Denote by $\Gamma_0,\Gamma_1,\ldots,\Gamma_p$ the connected components of $\partial\Omega$. For every $i \in \{0,1,\ldots,p\}$, we define $V'_{\partial,i}$ as the set of vertices in $V'_{\partial}$ belonging to $\Gamma_i$, and $E'_{\partial,i}$ as the set of dual edges $\{\mb{v}',\mb{w}'\}$ in $E'_\partial$ such that $\{\mb{v}',\mb{w}'\} \subset \Gamma_i$. Indicate by $\A'_i$ the graph $(V'_{\partial,i},E'_{\partial,i})$. It is the graph induced by $\A'$ on $\Gamma_i$.

\begin{defin} \label{def:sette}
Let $\B'=(V' \cup V'_\partial,N')$ be a spanning tree of $\A'$. We say that $\B'$ is a \emph{Seifert dual (barycentric) spanning tree of $\T$} if it restricts to a spanning tree on each connected component $\Gamma_i$ of $\partial\Omega$; more precisely, if
\begin{equation} \label{eq:seifert-dst}
\text{$(V'_{\partial,i},N' \cap E'_{\partial,i})$ is a spanning tree of $\A'_i$ for every $i \in \{0,1,\ldots,p\}$.}
\end{equation}
\end{defin}

\begin{remark} \label{rem:intro}
We pointed out in the introduction that, given a spanning tree $\B'$ of $\A'$, the number $N_{\B'}$ of oriented faces of $\T$ whose dual edge belongs to $\B'$ is $\geq \mk{t}+p$, where $\mk{t}$ is the number of tetrahedra of $\T$. Moreover, the equality holds if and only if $\B'$ is a Seifert dual spanning tree of $\T$. The following simple argument of graph theory explains why. Let $i \in \{0,1,\ldots,p\}$. Indicate by $\mk{v}'_i$ the number of vertices of $\A'_i$ or, equivalently, the number of faces of $F_\partial$ contained in $\Gamma_i$. Evidently, the number of vertices of $\A'$ is $\mk{t}+\sum_{i=0}^p\mk{v}'_i$. Denote by $\B'_i$ the graph induced by $\B'$ on $\Gamma_i$ and by $k_i$ the number of connected components of $\B'_i$. Bearing in mind that $\B'$ is a spanning tree of $\A'$, we infer at once that $\B'_i$ is a subgraph of $\A'_i$ with the same vertices of $\A'_i$, whose connected components are trees. In particular, $\B'_i$ is a spanning tree of $\A'_i$ if and only if $k_i=1$. Since in a finite tree the number of edges is equal to the number of vertices minus $1$, we have that the number of edges of $\B'$ is $(\mk{t}+\sum_{i=0}^p\mk{v}'_i)-1$ and the number of edges of $\B'_i$ is $\mk{v}'_i-k_i$. It follows that
\[
\textstyle
N_{\B'}=(\mk{t}+\sum_{i=0}^p\mk{v}'_i)-1-\sum_{i=0}^p(\mk{v}'_i-k_i)=\mk{t}-1+\sum_{i=0}^pk_i \geq \mk{t}+p
\]
and $N_{\B'}=\mk{t}+p$ if and only if each $k_i$ is equal to $1$ or, equivalently, if and only if the graph $\B'_i$ is a spanning tree of $\A'_i$ for each $i \in \{0,1,\ldots,p\}$; namely, if $\B'$ is a Seifert dual spanning tree of $\T$.

As we have just said in the introduction, we are mainly interested in Seifert dual spanning tree of $\T$ because $Z_2(\T;\Z)=\ker(\partial_2)$ is a free abelian group of rank $\mk{t}+p$. Let us explain the latter assertion. Since $H_3(\T;\Z)$ is trivial, the boundary operator $\partial_3$ is injective. It follows immediately that $B_2(\T;\Z)$ is a free abelian group of rank $\mk{t}$ and the boundaries of tetrahedra $t_1,\ldots,t_{\mk{t}}$ of $\T$ furnish one of its basis. For every $i \in \{1,\ldots,p\}$, denote by $\gamma_i$ the $2$-cycle in $Z_2(\T;\Z)$ associated with the triangulation of $\Gamma_i$ induced by $\T$. It is well known that $H_2(\T;\Z)$ is a free abelian group of rank $p$ and the homology classes of the $\gamma_i$'s form one of its basis. Bearing in mind that $Z_2(\T;\Z)$ is isomorphic to $B_2(\T;\Z) \oplus H_2(\T;\Z)$, we infer that $Z_2(\T;\Z)$ is a free abelian group of rank $\mk{t}+p$ and $\{\partial_3t_1,\ldots,\partial_3t_{\mk{t}},\gamma_1,\ldots,\gamma_p\}$ is a basis of $Z_2(\T;\Z)$.
\end{remark}

The reader observes that a Seifert dual spanning tree of $\T$ always exists and it is easy to construct. Indeed, it suffices to choose a spanning tree $\B'_i$ of each $\A'_i$ and to extend the union of the $\B'_i$'s to a spanning tree of the whole $\A'$.

Our main result reads as follows:

\begin{theo} \label{thm:main}
Let $\B'=(V' \cup V'_\partial,N')$ be a Seifert dual spanning tree of $\T$ and let $\mc{N}'$ be its set of oriented dual edges. Then, for every $1$-boundary $\gamma$ of $\T$, there exists, and is unique, a homological Seifert surface $S=\sum_{f \in \F}b_ff$ of $\gamma$ in $\T$ such that $b_f=0$ for every $f \in \F$ with $D(f) \in \mc{N}'$.
Moreover, it holds:
\begin{equation} \label{eq:formula}
b_f=\lk\big(R_+(\gamma),\cb(D(f))\big).
\end{equation}
for every $f \in \F$.
\end{theo}

We consider also the problem of the existence and of the construction of internal homolo\-gical Seifert surfaces. To this end, we need a definition, in which we will employ the notion of maximal plug-set of $\T$ introduced in Definition \ref{def:plug-set}.

\begin{defin}
Given a spanning tree $\B'=(V' \cup V'_\partial,N')$ of $\A'$, we say that $\B'$ is a \emph{strongly-Seifert dual (barycentric) spanning tree of $\T$} if it satisfies \eqref{eq:seifert-dst} and the set $N'$ of its edges contains a maximal plug-set of $\T$.
\end{defin}

Once again, strongly-Seifert dual spanning trees of $\T$ always exist, and are easy to construct. Let $i \in \{0,1,\ldots,p\}$. Choose a spanning tree $\B'_i=(V'_{\partial,i},N'_i)$ of each $\A'_i$. Denote by $\mr{J}_{\T\!,i}^r$ the set of regular plugs of $\T$ induced by the faces $f \in F_\partial \setminus F_\partial^\sangle$ with $|f| \subset \Gamma_i$. Let $K_{\partial,i}$ be the set of tetrahedra $t \in K$ such that $t$ contains at least one face in $\Gamma_i$ and let $K_{\partial,i}^\sangle:=K_{\partial,i} \cap K_\partial^\sangle$. For every $t \in K_{\partial,i}^\sangle$, choose one of the corner faces of $\T$ contained in $t$ and denote it by $f_{t,i}^\sangle$. Let $J'_i$ be the set of corner plugs of $\T$ induced by the chosen corner faces $\{f_{t,i}^\sangle\}_{t \in K_{\partial,i}^\sangle}$, let $J''_i:=\mr{J}_{\T\!,i}^r \cup J'_i$ and let $V_i''$ be the set of dual vertices of $\T$ of the form $B(t)$ with $t \in K_{\partial,i}$; namely, $V''_i=\{B(t) \in V' \, | \, t \in K_{\partial,i}\}$. By construction, the graph $\B''_i:=(V'_{\partial,i} \cup V''_i,N'_i \cup J''_i)$ is a tree containing $\B'_i$. Moreover, it is immediate to verify that, for every $i,j \in \{0,1,\ldots,p\}$ with $i \neq j$, $\B''_i$ and $\B''_j$ have neither vertices nor edges in common. In particular, the set $\bigcup_{i=0}^pJ''_i$ is a maximal plug-set of $\T$. Now one can extend the union of the $\B''_i$'s to a spanning tree of $\A'$, which turns out to be a strongly-Seifert dual spanning tree of $\T$.

The reader observes that the maximal plug-set of $\T$ contained in the set of edges of a given strongly-Seifert dual spanning tree of $\T$, which exists by definition, is unique.

As a consequence of Theorem \ref{thm:main}, we have the following result, which settles the above-mentioned problem of the existence and of the construction of internal homological Seifert surfaces.

\begin{theo} \label{thm:internal}
The following assertions hold.
\begin{itemize}
 \item[$(\mr{i})$] A $1$-boundary of $\T$ has an internal homological Seifert surface in $\T$ if and only if it is corner-free.  
 \item[$(\mr{ii})$] Let $\B'=(V' \cup V'_\partial,N')$ be a strongly-Seifert dual spanning tree of $\T$ and let $\mc{N}'$ be its set of oriented dual edges. Then, for every corner-free $1$-boundary $\gamma$ of $\T$, there exists, and is unique, an internal homological Seifert surface $S=\sum_{f \in \F}b_ff$ of $\gamma$ in $\T$ such that $b_f=0$ for every $f \in \F$ with $D(f) \in \mc{N}'$. Moreover, each coefficient $b_f$ satisfies formula \eqref{eq:formula}.
 \end{itemize}
\end{theo}

In particular, we have:

\begin{cor} \label{cor:internal}
The following assertions hold.
\begin{itemize}
 \item[$(\mr{i})$] Every internal $1$-boundary of $\T$ has an internal homological Seifert surface in $\T$.
 \item[$(\mr{ii})$] If $\T$ is the first barycentric subdivision of some triangulation of $\overline\Omega$, then every $1$-boundary of $\T$ has an internal homological Seifert surface in $\T$.
\end{itemize}
\end{cor}


\subsection{The proofs}


We begin by proving Theorem \ref{thm:main}. First, we need three preliminary lemmas.

Let $\B'=(V' \cup V'_\partial,N')$ be a Seifert dual spanning tree of $\T$ and let $\mc{N}'$ be its set of oriented dual edges. We define $\G:=\{f \in \F \, | \, D(f) \not\in \mc{N}'\}$ and, for every $f \in \F$, we simplify the  notation by writing $\sigma(f)$ in place of $\cb(D(f))$.

\begin{lemma} \label{lem:dual-coil}
For every $f,g \in \G$, it holds:
\[
\lk\big(\partial_2f,R_-(\sigma(g))\big)=
\left\{
\begin{array}{ll}
1 & \text{if $f =g$} \\
0 & \text{if $f \neq g$}
\end{array}.
\right.
\]
\end{lemma}
\begin{proof}
Let $f,g \in \mc{G}$ and let $\vv',\mb{w}' \in V' \cup V'_\partial$ such that $\partial_1D(g)=\vv'-\mb{w}'$. By definition of $\sigma(g)$, there exist, and are unique, an integer $\ell \geq 2$, a $\ell$-upla of pairwise disjoint vertices $(p'_0,p'_1,\ldots,p'_\ell)$ of $V' \cup V'_\partial$ and, for every $i \in \{1,\ldots,\ell\}$, $\delta_i \in \{-1,1\}$ and $e'_i \in \mc{N}'$ such that $p'_0=\vv'$, $p'_\ell=\mb{w}'$,  $\partial_1(\delta_ie'_i)=p'_i-p'_{i-1}$ for every $i \in \{1,\ldots,\ell\}$ and $\sigma(g)=D(g)+\sum_{i=1}^\ell\delta_ie'_i$.

There are only two cases in which the intersection $|f| \cap |R_-(\sigma(g))|$ is non-empty, and hence the linking number $\lk(\partial_2f,R_-(\sigma(g)))$ may be different from zero.

\textit{Case 1: Assume $f=g$}. In this case, we have that $|f| \cap |R_-(\sigma(g))|=\{B(f)\}$. We must prove that $\lk(\partial_2f,R_-(\sigma(g)))=1$. Suppose that $f \not\in \F_\partial$. Observe that the intersection between $f$ and $R_-(\sigma(g))$ is not transverse, because $D(g)=[\mb{w}',B(f)]+[B(f),\vv']$. Let $a'_1$ be a point of the segment $|[\mb{w}',B(f)]|$ different from $B(f)$, let $b'_1$ be a point of the segment $|[B(f),\vv']|$ different from $B(f)$ and let $\widehat{\gamma}_1$ be the $1$-cycle of $\R^3$ defined by setting
\[
\widehat{\gamma}_1:=[\mb{w}',a'_1]+[a'_1,b'_1]+[b'_1,\vv']+\sum_{i=1}^\ell\delta_ir_-(e'_i)\, ,
\]
see Figure~\ref{fig:Tintersection} on the left.
If $a'_1$ and $b'_1$ are chosen sufficiently close to $B(f)$, we have that $\widehat{\gamma}_1$ is homologous to $R_-(\sigma(g))$ in $\R^3 \setminus |\partial_2f|$, it intersects $f$ transversally in one point belonging to $|[a'_1,b'_1]| \setminus \{a'_1,b'_1\}$ and $\sign(\bs{\tau}([a'_1,b'_1]) \cdot \bs{\nu}(f))=1$. By the definition of linking number, we infer that  $\lk(\partial_2f,R_-(\sigma(g)))=1$.

Suppose now that $f \in \F_\partial$. Changing the orientation of $f$ if necessary, we may also suppose that $\vv'=B(f)$. It follows that $p'_1$ is the barycenter of an oriented face $f_1$ in $\E_\partial$ having an (oriented) edge $e$ in common with $f$ and hence $\delta_1r_-(e'_1)=[\vv',\mb{x}_{e'_1}]+[\mb{x}_{e'_1},p'_1]$ for some point $\mb{x}_{e'_1} \in \R^3 \setminus \overline{\Omega}$ close to $B(e)$ (see Subsection \ref{subsec:lk-recogn-retract} for the definition of $r_-$). Let us proceed as above. Choose a point $a'_2 \in |[\mb{w}',\vv']| \setminus \{\vv'\}$ close to $\vv'$ and a point $b'_2 \in |[\vv',\mb{x}_{e'_1}]| \setminus \{\vv'\}$ close to $\vv'$. Then the $1$-cycle $\widehat{\gamma}_2$ of $\R^3$ defined by setting
\[
\widehat{\gamma}_2:=[\mb{w}',a'_2]+[a'_2,b'_2]+[b'_2,\mb{x}_{e'_1}]+[\mb{x}_{e'_1},p'_1]+\sum_{i=2}^\ell\delta_ir_-(e'_i)\, ,
\]
see Figure~\ref{fig:Tintersection} on the right,
is homologous to $R_-(\sigma(g))$ in $\R^3 \setminus |\partial_2f|$, it intersects $f$ transversally in one point belonging to $|[a'_2,b'_2]| \setminus \{a'_2,b'_2\}$ and $\sign(\bs{\tau}([a'_2,b'_2]) \cdot \bs{\nu}(f))=1$. It follows that $\lk(\partial_2f,R_-(\sigma(g)))=1$, as desired.
\begin{figure}[!htb]
\centering
 \includegraphics[width=\textwidth]{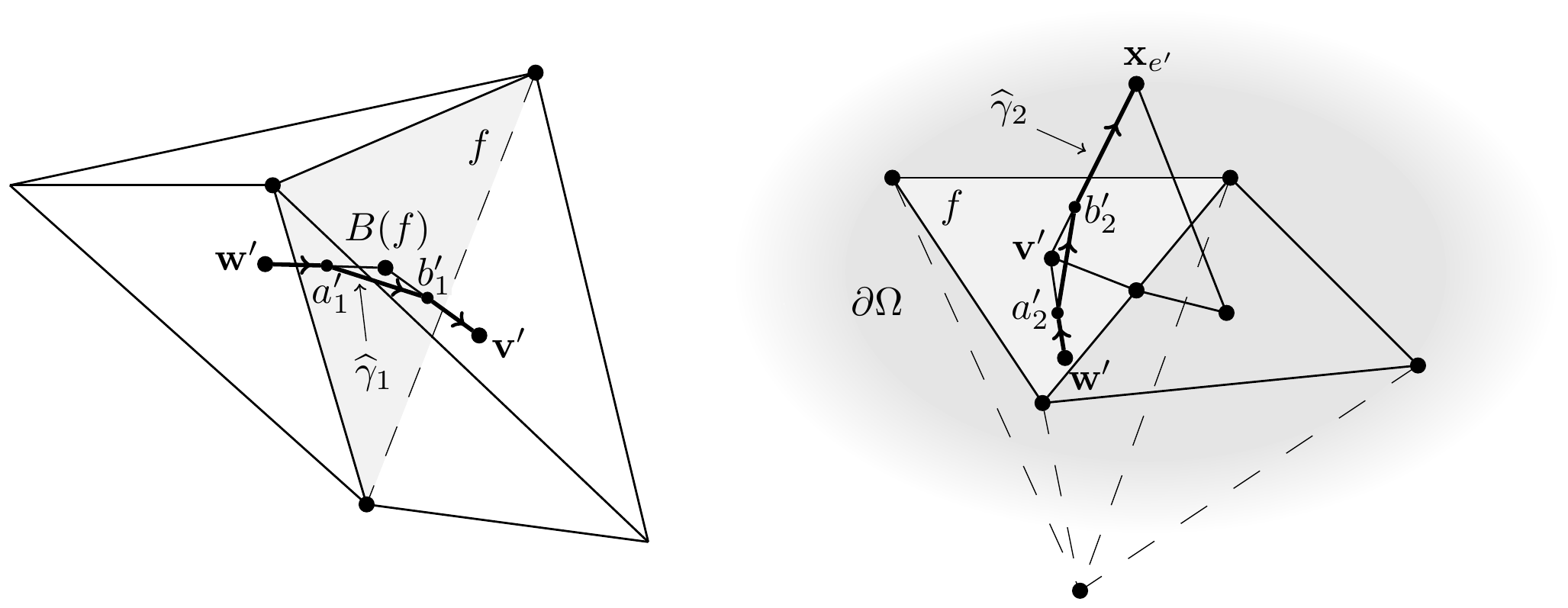}
  \caption{The $1$-cycles $\widehat{\gamma}_1$ (on the left) and $\widehat{\gamma}_2$ (on the right).} \label{fig:Tintersection}
\end{figure}

\textit{Case 2. Assume that $f \neq g$, $f \in \F_\partial$ and there exists $h \in \{1,\ldots,\ell-1\}$ such that $p'_h=B(f)$ and both  $e'_h$ and $e'_{h+1}$ belong to $\E'_\partial$.} We know that $\delta_hr_-(e'_h)=[p'_{h-1},\mb{x}_{e'_h}]+[\mb{x}_{e'_h},p'_h]$ and $\delta_{h+1}r_-(e'_{h+1})=[p'_h,\mb{x}_{e'_{h+1}}]+[\mb{x}_{e'_{h+1}},p'_{h+1}]$ for some  $\mb{x}_{e'_h},\mb{x}_{e'_{h+1}} \in \R^3 \setminus \overline{\Omega}$. In particular, it holds:
\[
R_-(\sigma(g))=c+[p'_{h-1},\mb{x}_{e'_h}]+[\mb{x}_{e'_h},p'_h]+[p'_h,\mb{x}_{e'_{h+1}}]+[\mb{x}_{e'_{h+1}},p'_{h+1}],
\]
where $c:=D(g)+\sum_{i \in \{1,\ldots,\ell\} \setminus \{h,h+1\}}\delta_ir_-(e'_i)$. Let $a'_3 \in |[\mb{x}_{e'_h},p'_h]| \setminus \{p'_h\}$, let $b'_3  \in |[p'_h,\mb{x}_{e'_{h+1}}]| \setminus \{p'_h\}$ and let $\widehat{\gamma}_3$ be the $1$-cycle of $\R^3$ defined by setting
\[
\widehat{\gamma}_3:=c+[p'_{h-1},\mb{x}_{e'_h}]+[\mb{x}_{e'_h},a'_3]+[a'_3,b'_3]+[b'_3,\mb{x}_{e'_{h+1}}]+[\mb{x}_{e'_{h+1}},p'_{h+1}], ,
\]
see Figure~\ref{fig:gamma3}.
If $a'_3$ and $b'_3$ are chosen sufficiently close to $p'_h$, then $\widehat{\gamma}_3$ is homologous to $R_-(\sigma(g))$ in $\R^3 \setminus |\partial_2f|$ and it does not intersects $|f|$. It follows that $\lk(\partial_2f,R_-(\sigma(g)))=0$.
\begin{figure}[!htb]
\centering
 \includegraphics[width=.55\textwidth]{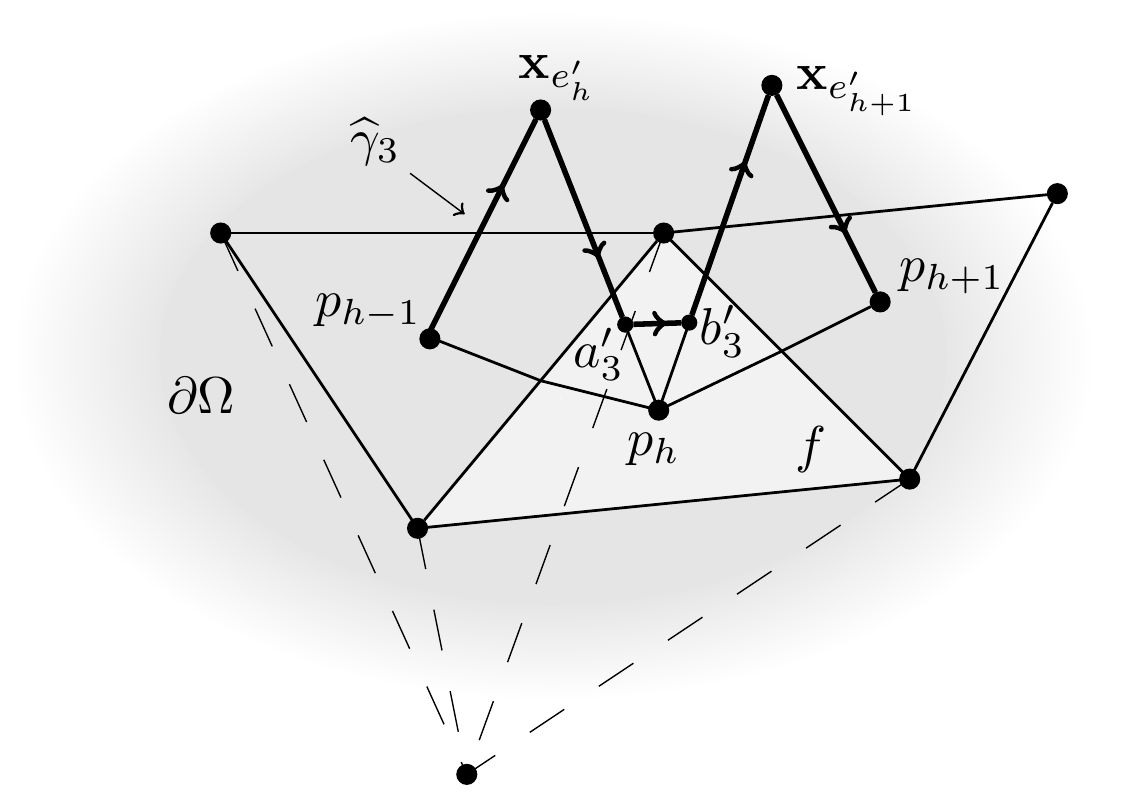}
  \caption{The $1$-cycle $\widehat{\gamma}_3$.} \label{fig:gamma3}
\end{figure}

This completes the proof.
\end{proof}

\begin{lemma} \label{lem:coil}
Let $\xi=\sum_{e \in \E}\alpha_ee$ be a $1$-cycle of $\T$. Then, for every $e^* \in \E$, it holds:
\begin{equation} \label{eq:coil}
\lk\big(\xi,R_-(\coil(e^*))\big)=\alpha_{e^*}.
\end{equation}
In particular, $\xi=0$ if and only if $\lk\big(\xi,R_-(\coil(e^*))\big)=0$ for every $e^* \in \E$.
\end{lemma}
\begin{proof}
Fix $e^* \in \E$, a spanning tree $(V,L)$ of the graph $(V,E)$ such that $|\partial_1 e^*| \not \in L$ and a vertex $\mb{a} \in V$, we consider as a root of $(V,L)$. Denote by $\mc{L}$ the set of oriented edges in $\E$ determined by the corresponding edges in $L$; namely, $\mc{L}:=\big\{e \in \E \, \big| \, |\partial_1e| \in L\big\}$. For every $\mb{v} \in V$, denote by $C_{\mb{v}}$ the (unique) $1$-chain of $\T$ such that $|C_{\mb{v}}| \subset \bigcup_{e \in \mc{L}}|e|$ and $\partial_1 C_{\bf v}={\bf v}-{\bf a}$. Given $e=[\mb{a}_e,\mb{b}_e] \in \E$, we denote by $\sigma_e$ the $1$-cycle of $\T$ given by $\sigma_e:=C_{\mb{a}_e}+e-C_{\mb{b}_e}$.

By hypothesis, $\xi$ is a $1$-cycle of $\T$ and hence $0 = \partial_1 \xi=\sum_{e \in \E}\alpha_e(\mb{b}_e-\mb{a}_e)$ in $C_0(\T;\Z)$. It follows that $\sum_{e \in \E}\alpha_e(C_{\mb{b}_e}-C_{\mb{a}_e})=0$ in $C_1(\T;\Z)$ as well. In this way, we obtain that
\[
\sum_{e \in \E} \alpha_e \sigma_e= \sum_{e \in \E} \alpha_e (C_{\mb{a}_e}+e-C_{\mb{b}_e})=\xi-\sum_{e \in \E} \alpha_e (C_{\mb{b}_e}-C_{\mb{a}_e})=\xi.
\]
Then
\[
\lk\big(\xi,R_-(\coil(e^*))\big)=\sum_{e \in \E} \alpha_e \, \lk\big(\sigma_e,R_-(\coil(e^*))\big).
\]
Thanks to the latter equality, it suffices to show that
\[
\lk\big(\sigma_e,R_-(\coil(e^*))\big)
=\left\{
\begin{array}{ll}
1 & \text{if $e=e^*$} \\
0 & \text{if $e \neq e^*$}
\end{array}.
\right.
\]
To do this, we use an argument similar to the one employed in the proof of the preceding lemma. However, contrarily to such a proof, we omit the details concerning the construction of ``small deformations of $\sigma_e$'' to obtain trasversality. If $e \in \mc{L}$, then $e \neq e^*$ (because $e^* \not\in \mc{L}$), $\sigma_e=0$ and hence $\lk\big(\sigma_e,R_-(\coil(e^*))\big)=0$. If $e \not\in \mc{L} \cup \{e^*\}$, then $|\sigma_e| \cap |D(e^*)|=\emptyset$, so $\lk\big(\sigma_e,R_-(\coil(e^*))\big)=0$. Suppose $e=e^* \in \E \setminus \E_\partial$. In this case, we have that $R_-(\coil(e))=\coil(e)=\partial_2D(e)$ and $|\sigma_e| \cap |D(e)|=\{B(e)\}$. By \eqref{eq:def-lk}, it follows immediately that $\lk\big(\sigma_e,R_-(\coil(e))\big)=\pm 1$. The sign of such a linking number is positive, because the triangles forming $D(e)$ was oriented by $e$ via the right hand rule. Finally, consider the case in which $e=e^* \in \E_\partial$. By construction (see Definition \ref{def:coil} and points \eqref{eq:S'} and \eqref{eq:R_-}), we have that $R_-(\coil(e))=\partial_2\big(D(e)+S'_{D_\partial(e)}\big)$ and $|\sigma_e| \cap \big|D(e)+S'_{D_\partial(e)}\big|=\{B(e)\}$. Once again, we infer that $\lk\big(\sigma_e,R_-(\coil(e))\big)=1$.
\end{proof}

\begin{lemma} \label{lem:boundary}
Let $\gamma$ be a $1$-boundary of $\T$. Then, for every $e' \in \E'_\partial$, it holds:
\[
\lk\big(\gamma,R_-(\cb(e'))\big)=0.
\]
\end{lemma}
\begin{proof}
If $e' \in \mc{N}'$, then $\cb(e')=0$ and the result is trivial. Choose $e' \in \E'_\partial \setminus \mc{N}'$ and indicate by $i$ the unique index in $\{0,1,\ldots,p\}$ such that $|\partial_1e'| \in E'_{\partial,i}$ or, equivalently, $|e'| \subset \Gamma_i$. Since $\B'_{\partial,i}:=(V'_{\partial,i},N' \cap E'_{\partial,i})$ is a spanning tree of $\A'_i$, there exists a unique vertex $\mb{b}'_i$ in $V'_{\partial,i}$ such that $|C'_{\mb{b}'_i}| \subset \bigcup_{e' \in \E'}|e'|$; namely, in the expression of $C'_{\mb{b}'_i}$, the oriented dual edges in $\E'_\partial$ appear with null coefficients (see \eqref{eq:C'} for the definition of $C'_{\mb{b}'_i}$). Let $\E'_{\partial,i}$ be the set of oriented dual edges in $\E'_\partial$ corresponding to the edges in $E'_{\partial,i}$; namely, $\E'_{\partial,i}:=\big\{e' \in \E'_\partial \, \big| \, |\partial_1e'| \in E'_{\partial,i}\big\}$. For every $\mb{v}' \in V'_{\partial,i}$, denote by $c'_{i,\mb{v}'}$ the unique $1$-chain of $\B'_{\partial,i}$ from $\mb{b}'_i$ to $\mb{v}'$. Let $e' \in \E'_{\partial,i}$ with $\partial_1e'=\mb{v}'-\mb{w}'$. Observe that $C'_{\mb{v}'}=C'_{\mb{b}'_i}+c'_{i,\mb{v}'}$, $C'_{\mb{w}'}=C'_{\mb{b}'_i}+c'_{i,\mb{w}'}$ and hence
\[
\cb(e')=c'_{i,\mb{w}'}+e'-c'_{i,\mb{v}'}.
\]
It follows that $|\cb(e')| \subset \Gamma_i$ and hence $|R_-(\cb(e'))| \subset (\R^3 \setminus \overline{\Omega}) \cup V'_{\partial,i}$. Since $\partial\Omega$ has a collar in $\R^3 \setminus \Omega$, it is easy to find a $1$-cycle $\eta$ of $\R^3$ such that $|\eta| \subset \R^3 \setminus \overline{\Omega}$ and $\eta$ is homologous to $R_-(\cb(e'))$ in $(\R^3 \setminus \overline{\Omega}) \cup V'_{\partial,i} \subset \R^3 \setminus |\gamma|$. Thanks to \eqref{eq:homol-inv}, we infer that $\lk\big(\gamma,R_-(\cb(e'))\big)=\lk(\gamma,\eta)$. On the other hand, by hypothesis, $\gamma$ bounds in $\overline{\Omega}$. Since $\overline{\Omega} \subset \R^3 \setminus |\eta|$, $\gamma$ bounds in $\R^3 \setminus |\eta|$ as well. Equality \eqref{eq:bounds} ensures that $\lk(\gamma,\eta)=0$, as desired.
\end{proof}

We are now in position to prove our results.

\begin{proof}[Proof of Theorem~\ref{thm:main}]
We start by proving the uniqueness of solution. Suppose that $S=\sum_{f \in \F}b_ff$ is a homological Seifert surface of $\gamma$ in $\T$ such that $b_f=0$ for every $f$ with $D(f) \in \cal N'$; namely, for every $f \in \F \setminus \G$. We must show that $b_f=\lk\big(R_+(\gamma),\sigma(f)\big)$ for every $f \in \G$. The reader observes that, if $f \in \F \setminus \G$, then $\sigma(f)=0$ and hence $\lk\big(R_+(\gamma),\sigma(f)\big)$ is automatically equal to $0=b_f$. Choose $f^* \in \G$. By Lemma~\ref{lem:C+C-}, we infer that
\begin{align*}
\lk\big(R_+(\gamma),\sigma(f^*)\big) &=
\textstyle
\lk\big(\gamma,R_-(\sigma(f^*))\big)=\lk\!\left(\sum_{f \in \G}b_f\partial_2f,R_-(\sigma(f^*))\right)=\\
&=
\textstyle
\sum_{f \in \G}b_f \, \lk\big(\partial_2f,R_-(\sigma(f^*))\big)
.
\end{align*}
Now Lemma~\ref{lem:dual-coil} implies that
\[
\sum_{f \in \G}b_f \,
\lk\big(\partial_2f,R_-(\sigma(f^*))\big)=b_{f^*}.
\]
In this way, we have that $\lk\big(R_+(\gamma),\sigma(f^*)\big)=b_{f^*}$ for every $f^* \in \G$, as desired.

It remains to prove that, if $b_f:=\lk\big(R_+(\gamma),\sigma(f)\big)$ for every $f \in \G$, then the boundary of the $2$-chain $S:=\sum_{f \in \G}b_ff$ of $\T$ is equal to $\gamma$. This is equivalent to show that the $1$-cycle $\eta:=\gamma-\partial_2S=\gamma-\sum_{f \in \mc{G}}b_f \, \partial_2f$ of $\T$ is equal to the zero $1$-chain of $\T$. Thanks to Lemma \ref{lem:coil}, this is in turn equivalent to show that $\lk\big(\eta,R_-(\coil(e))\big)=0$ for every $e \in \E$.

Fix $e \in \E$ and write $\coil(e)$ explicitly as follows:
\[
\coil(e)=\sum_{e' \in \E' \cup \E'_\partial} a'_{e'}e'
\]
for some (unique) integer $a'_{e'}$. For every $e' \in \E' \cup \E'_\partial$, denote by $\mb{v}'(e')$ and $\mb{w}'(e')$ the dual vertices in $V'$ such that $\partial_1e'=\mb{v}'(e')-\mb{w}'(e')$. Since $\coil(e)$ is a $1$-cycle of $\A'$ (a $1$-boundary of $\A'$ indeed), we have that $0=\partial_1\coil(e)=\sum_{e' \in \E' \cup \E'_\partial}a'_{e'}(\mb{v}'(e')-\mb{w}'(e'))$. It follows that $\sum_{e' \in \E' \cup \E'_\partial}a'_{e'}(C'_{\mb{v}'(e')}-C'_{\mb{w}'(e')})=0$ as well, and hence
\begin{equation} \label{eq:sigma}
\coil(e)=\sum_{e' \in \E' \cup \E'_\partial} a'_{e'}e'-\sum_{e' \in \E' \cup \E'_\partial}a'_{e'}(C'_{\mb{v}'(e')}-C'_{\mb{w}'(e')})=\sum_{e' \in \E' \cup \E'_\partial} a'_{e'}\cb(e').
\end{equation}

In this way, in order to complete the proof, it suffices to prove that
\begin{center}
$\lk\big(\eta,R_-(\cb(e')\big)=0$ for every $e' \in \E' \cup \E'_\partial$.
\end{center}

We distinguish three cases: $e'\in \mc{N}'$, $e' \in \E'\setminus \mc{N}'$ and $e' \in \E'_\partial \setminus \mc{N}'$.

If $e'\in \mc{N}'$, then $\cb(e')=0$ and hence $\lk\big(\eta,R_-(\cb(e'))\big)=0$.

If $e' \in \E' \setminus \mc{N}'$, then $e'=D(f^*)$ for some (unique) $f^* \in \mc{G}$. Bearing in mind Lemma~\ref{lem:dual-coil}, we obtain:
\begin{align*}
\lk\big(\eta,R_-(\cb(e'))\big) &=\lk\big(\eta,R_-(\sigma(f^*))\big)=\\ &=\lk\big(\gamma,R_-(\sigma(f^*))\big) - \sum_{f \in \mc{G}} b_f \, \lk\big(\partial_2f,R_-(\sigma(f^*))\big)=\\
&=b_{f^*}-b_{f^*}=0.
\end{align*}

Finally, if $e' \in \E'_\partial \setminus \mc{N}'$, then Lemma~\ref{lem:boundary} ensures that $\lk\big(\eta,R_-(\cb(e'))\big)=0$, because $\eta$ is a $1$-boundary of $\T$.
\end{proof}

We conclude with the proofs of Theorem \ref{thm:internal} and of its Corollary \ref{cor:internal}.

\begin{proof}[Proof of Theorem~\ref{thm:internal}]
Let $\gamma$ be a $1$-boundary of $\T$. It is evident that the boundary of any internal $2$-chain of $\T$ cannot contain oriented edges determined by corner edges of $\T$. Hence if $\gamma$ admits an internal homological Seifert surface in $\T$, then it must be corner-free.

Suppose $\gamma$ is corner-free. Let $\B'=(V' \cup V'_\partial,N')$ and $\mc{N}'$ be as in the statement of point $(\mr{ii})$, and let $J$ be the maximal plug-set of $\T$ contained in $N'$. Write $J$ as in Remark~\ref{maximal-plug-set}: $J=\mr{J}_{\T}^{\mr{r}} \cup J'$, where $J'$ is the set of corner plugs of $\T$ belonging to $J$. Denote by $F^\sangle$ the set of corner faces of $\T$ inducing the corner plugs in $J'$.

By Theorem~\ref{thm:main}, there exists, and is unique, a homological Seifert surface $S=\sum_{f \in \F}b_ff$ of $\gamma$ in $\T$ such that $b_f=0$ for every $f \in \F$ with $D(f) \in \mc{N}'$. Moreover, each $b_f$ satisfies formula \eqref{eq:formula}.

We must prove that $S$ is internal; namely, $b_f=0$ for every $f \in \F_\partial$. Since $J \subset N'$, it suffices to show the following: if $g$ is an oriented face in $\F_\partial$ such that the corresponding (non-oriented) face belongs to $F_\partial^\sangle \setminus F^\sangle$, then $b_g=0$. Let $g$ be such an oriented face in $\F_\partial$. Then there exist vertices $\mb{v},\mb{w},\mb{z}^*,\mb{z}^{**} \in V_\partial \cap \Gamma_i$ for some (unique) $i \in \{0,1,\ldots,p\}$ such that the tetrahedron  $\{\mb{v},\mb{w},\mb{z}^*,\mb{z}^{**}\}$ of $\T$ is a corner tetrahedron, its face $\{\mb{v},\mb{w},\mb{z}^*\}$ belongs to $F^\sangle$ and the oriented face in $\F$ corresponding to $\{\mb{v},\mb{w},\mb{z}^{**}\}$ is equal to $g$. Indicate by $f$ the oriented face in $\F$ corresponding to $\{\mb{v},\mb{w},\mb{z}^*\}$, by $e$ the oriented edge in $\E_\partial$ corresponding to $\{\vv,\mb{w}\}$, by $e'$ the oriented dual edge $D_\partial(e)$ in $\E'_\partial$ and by $\vv',\mb{w}'$ the vertices in $V'_\partial$ such that $\partial_1(e')=\vv'-\mb{w}'$. Observe that there exist, and are unique, $s_1,s_2 \in \{-1,1\}$ such that
\begin{equation} \label{eq:ce}
\coil(e)=e'+s_1D(f)+s_2D(g).
\end{equation}
In particular, since $\partial_1(\coil(e))=0$, we have:
\begin{equation} \label{eq:v'w'}
\vv'-\mb{w}'=\partial_1(-s_1D(f)-s_2D(g)).
\end{equation}

By hypothesis, $\B'_{\partial,i}:=(V'_{\partial,i},N' \cap E'_{\partial,i})$ is a spanning tree of $\A'_i$. In this way, there exists a unique $1$-chain $C$ in $\B'_{\partial,i}$ such that $\partial_1(C)=\mb{w}'-\vv'$. It follows that $\cb(e')=e'+C$. Moreover, by combining \eqref{eq:v'w'} with the fact that $D(f) \in \mc{N}'$, we infer at once that
\[
\sigma(g)=-s_2(-s_1D(f)-s_2D(g)+C)=D(g)+s_1s_2D(f)-s_2C.
\]
On the other hand, by \eqref{eq:ce}, we have also that $-s_1D(f)-s_2D(g)=e'-\coil(e)$ and hence
\begin{equation} \label{eq:sg}
\sigma(g)=-s_2(e'-\coil(e)+C)=-s_2\big(\cb(e')-\coil(e)\big)=-s_2\cb(e')+s_2\coil(e).
\end{equation}

By Lemma \ref{lem:boundary}, we know that $\lk(\gamma,R_-(\cb(e')))=0$. Moreover, since $\gamma$ is corner-free and $e \in \E_\partial^\sangle$, Lemma~\ref{lem:coil} ensures that $\lk(\gamma, R_-(\coil(e)))=0$. In this way, bearing in mind \eqref{eq:sg} and Lemma \ref{lem:C+C-}, we have:
\begin{align*}
b_g&=\lk(R_+(\gamma),\sigma(g))=-s_2\,\lk(R_+(\gamma),\cb(e'))+s_2\,\lk(R_+(\gamma),\coil(e))=\\
&=-s_2\,\lk(\gamma,R_-(\cb(e')))+s_2\,\lk(\gamma,R_-(\coil(e)))=0,
\end{align*}
as desired. This completes the proof.
\end{proof}

\begin{proof}[Proof of Corollary~\ref{cor:internal}]
$(\mr{i})$ An internal $1$-boundary of $\T$ is corner-free and hence it has an internal homological Seifert surface in $\T$ by Theorem \ref{thm:internal}.

$(\mr{ii})$ As above, this point follows immediately from Theorem \ref{thm:internal}. Indeed, if $\T$ is the first barycentric subdivision of some triangulation of $\overline\Omega$, then $K_\partial^\sangle=\emptyset$ and hence every $1$-boundary of $\T$ is corner-free.
\end{proof}


\section{An elimination algorithm} \label{sec:elim_alg}

Let $\gamma=\sum_{e \in \E}a_ee$ be a given $1$-boundary of $\T$. A $2$-chain $S=\sum_{f \in \F}b_ff$ of $\T$ is a homological Seifert surface of $\gamma$ in $\T$ if its coefficients $\{b_f\}_{f \in \F}$ satisfy the following equation in $C_1(\T;\Z)$:
\begin{equation} \label{eq:eq-hss}
\sum_{f \in \F}b_f\partial_2f=\sum_{e \in \E}a_ee.
\end{equation}
Let us write this equation more explicitly as a linear system with as many equations as edges and as many unknowns as faces of $\T$. Given $e \in \E$, let $\F(e)$ be the set $\big\{f \in \F \, \big| \, |e| \subset |f|\big\}$ of oriented faces in $\F$ incident on $e$ and let $\o_e:\F(e) \lra \{-1,1\}$ be the function sending $f \in \F(e)$ into the coefficient of $e$ in the expression of $\partial_2f$ as a formal linear combination of oriented edges in $\E$. Equation \eqref{eq:eq-hss} is equivalent to the linear system
\[
\sum_{f \in \F(e)}\o_e(f)b_f=a_e \quad \text{if $e \in \E$},
\]
where the unknowns $\{b_f\}_{f \in \F}$ are integers. Theorem \ref{thm:main} ensures that, if $\B'=(V' \cup V'_\partial,N')$ is a Seifert dual spanning tree of $\T$ and $\mc{N}'$ is its set of oriented dual edges, then the linear system
\begin{align}
& \textstyle \sum_{f \in \F(e)}\o_e(f)b_f=a_e \quad \text{if $e \in \E$} \label{eq:oe1}\\
&  b_f=0  \qquad \qquad \qquad \quad \; \; \; \text{if $D(f) \in \mc{N}'$} \label{eq:oe2}
\end{align}
has a unique solution given by the formula:
\begin{equation} \label{eq:G}
b_f=\lk\big(R_+(\gamma),\cb(D(f))\big)
\end{equation}
for every $f \in \G$, where $\G=\{f \in \F \, | \, D(f) \not\in \mc{N}'\}$.

As we have just recalled in the introduction, the linking number can be computed accurately. However, the use of formula \eqref{eq:G} is too expensive if $\T$ is fine. In fact, if $\mk{v}$ is the number of vertices of $\T$, $g$ is the first Betti number of $\overline{\Omega}$ and $\sharp\G$ is the cardinality of $\G$, then $\sharp\G$ is greater than or equal to $\frac{1}{2}\mk{v}+1-g$, which is usually huge if $\T$ is fine. Let us explain the latter assertion. Let $\mk{e}$, $\mk{f}$ and $\mk{t}$ be the numbers of edges, of faces and of tetrahedra of $\T$, respectively. Let us prove that $\sharp\G=\mk{e}-\mk{v}+1-g \geq \frac{1}{2}\mk{v}+1-g$. We know that $\sharp\G=\mk{f}-(\mk{t}+p)$ (see Remark~\ref{rem:intro}). The Euler characteristic $\chi(\T)=\mk{v}-\mk{e}+\mk{f}-\mk{t}$ of $\T$ is equal the sum $\sum_{j=0}^3(-1)^jr_j$, where $r_j$ is the rank of $H_j(\T;\Z)$. Since $r_0=1$, $r_1=g$, $r_2=p$ and $r_3=0$, we infer that $\mk{v}-\mk{e}+\mk{f}-\mk{t}=1-g+p$ and hence $\sharp\G=\mk{e}-\mk{v}+1-g$. Recall that, in a finite graph, the sum of degrees of its vertices equals two times the number of its edges. Apply this result to the graph $\A=(V,E)$. Since each vertex $v$ in $V$ belongs to at least one tetrahedron of $\T$, the degree of $v$, as a vertex of $\A$, is $\geq 3$. It follows that $\mk{e} \geq \frac{3}{2}\mk{v}$ and hence $\sharp\G \geq \frac{1}{2}\mk{v}+1-g$.

We present below a simple elimination algorithm that simplifies drastically the construction of homological Seifert surfaces given by Theorem \ref{thm:main}. Let us denote by $\mc{R}$ the set of oriented faces $f$ in $\F$ for which the corresponding coefficient $b_f$ is already known. Initially, thanks to \eqref{eq:oe2}, we have that $\mc{R}=\F \setminus \G$. If there exist edges $e$ such that exactly one oriented face $f^* \in \F(e)$ does not belong to $\mc{R}$; namely, if there exist equations of linear system \eqref{eq:oe1} with just one remaining unknown, then we compute the coefficients $b_{f^*}$ via such equations and update $\mc{R}$. If there are not such edges and $\mc{R} \ne \F$, then we pick an oriented face $f \in \F \setminus \mc{R}$, compute $b_f$ using explicit formula \eqref{eq:G} and update $\mc{R}$. More precisely, the algorithm reads as follows:
\begin{alg} \label{alg:main}
\begin{enumerate}
\item[]
  \item $\mc R:= \mc{F} \setminus \mc{G}$, $\mc D:=\E$.
  \item  while $\mc R \ne \F$
  \begin{enumerate}
    \item $n_{\mc R}:=card(\mc R)$
    \item for every $e \in \mc D$
    \begin{enumerate}
      \item if every oriented face of $\F(e)$ belong to $\mc R$
\begin{enumerate}
\item $\mc D=\mc D\setminus \{e\}$
\end{enumerate}
      \item if exactly one oriented face  $f^* \in \F(e)$ does not belong to $\mc R$
      \begin{enumerate}
        \item compute $b_f$ via \eqref{eq:oe1} 
        \item $\mc R=\mc R \cup \{ f \}$
        \item $\mc D=\mc D\setminus \{e\}$
      \end{enumerate}
    \end{enumerate}
    \item if $card(\mc R)=n_{\mc R}$
    \begin{enumerate}
      \item pick $f \not \in \mc R$ and compute $b_f=\lk(R_+(\gamma),\cb(D(f)))$
      \item $\mc R=\mc R \cup \{ f \}$
    \end{enumerate}
  \end{enumerate}
\end{enumerate}
\end{alg}

It is always possible to choose a Seifert dual spanning tree $\B'$ of $\T$ in such a way that, for some $e \in \E$, exactly one oriented face $f^* \in \F(e)$ does not belong to $\mc{N}'$. In fact, in all the numerical experiments we have considered, including knotted $1$-boundaries and homologically non-trivial computational domains, when we use breadth first spanning trees (BFS) \cite{cormen}, the elimination algorithm determines the homological Seifert surface directly, without computing any linking number.


\section{Numerical results} \label{sec:numerical}

Two different strategies for the construction of the Seifert dual spanning tree $\cal B'$ of $\T$ have been considered. In the first one, $\cal B'$ contains just one plug for each connected component of the boundary of $\Omega$, while, in the second one, $\cal B'$ contains a maximal plug-set $J$. Then, a spanning tree of the graph $(V', E')$, containing the selected plugs, is constructed in both cases by using a breadth first search (BFS) \cite{cormen} strategy.

The two strategies are now illustrated by means of a toy problem obtained by triangulating a cube, see Figure~\ref{fig:trivial}a. The first technique to construct a Seifert dual spanning tree $\mathcal{B}'$, denoted by BFS$_1$, consists of the following steps:
\begin{enumerate}
  \item Build a BFS spanning tree on each graph $\mathcal{A}_i'$ induced by $\mathcal{A}'$ on the connected component $\Gamma_i$ of $\partial\Omega$. We remark that this step is usually not required in practice as remarked later.
  \item Build an ``internal'' spanning tree of the graph $(V', E')$.
  \item For each $\Gamma_i$, add exactly one plug induced by a face in $\Gamma_i$.
\end{enumerate}
For the toy problem, a possible ``internal'' tree and the additional edge added at Step $3$ of the preceding procedure are represented in Figure \ref{fig:trivial}b. Given the $1$-boundary $\gamma$ represented in Figure \ref{fig:trivial}a by thicker edges, one can run the elimination algorithm Alg.~1, obtaining the $2$-chain $S_{BFS_1}$ whose support is depicted in Figure \ref{fig:trivial}c.

\begin{figure}[!htb]
\centering
 \includegraphics[width=\textwidth]{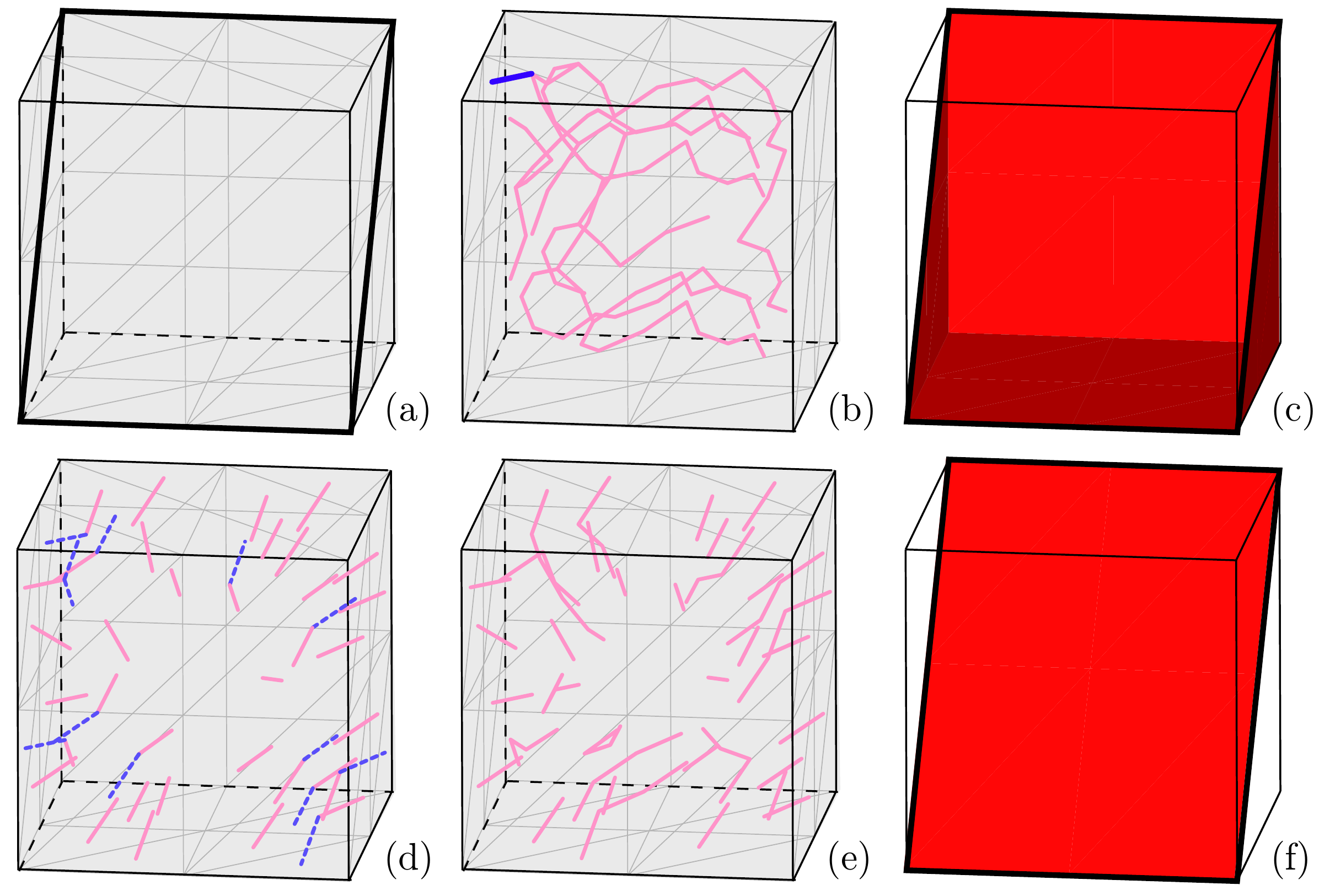}
  \caption{(a) A toy problem is obtained by triangulating a cube. Thicker edges represent the support of the $1$-boundary $\gamma$, whereas thin edges represent the edges of the triangulation of the cube contained in its boundary. (b) The Seifert dual spanning tree obtained with the BFS$_1$ technique (the tree in $\A'$ is not shown). The thicker dual edge represents the edge added at Step 3 of the algorithm. (c) The support of the $2$-chain obtained with the BFS$_1$ tree. (d) Continuous dual edges represent a maximal plug-set $J$, whereas the dotted dual edges are the plugs induced by corner faces that do not belong to  $J$. (e) The tree is completed in the interior of the triangulation by a BFS strategy (the tree in $\A'$ is not shown). (f) The support of the $2$-chain obtained with the BFS$_2$ tree.} \label{fig:trivial}
\end{figure}

\begin{figure}[!htb]
\centering
 \includegraphics[width=\textwidth]{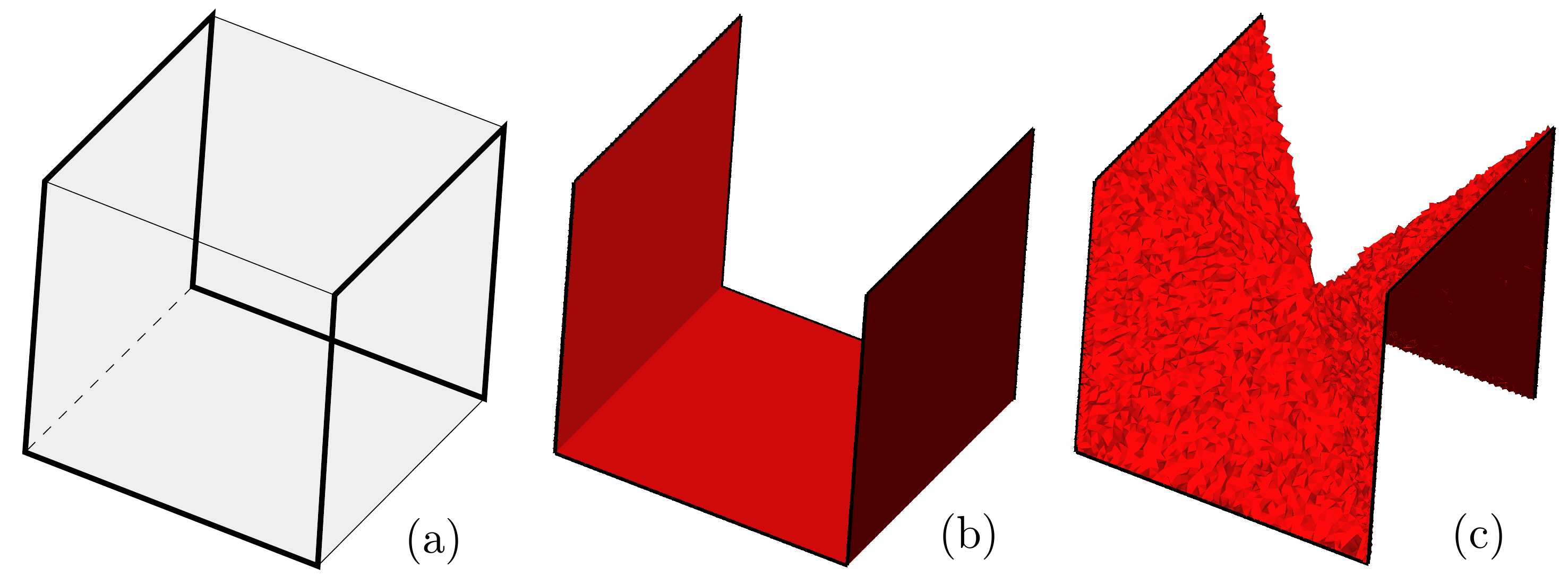}
  \caption{Toy problem $2$. (a) The $1$-boundary  $\gamma$ is represented by the thicker edges. (b) The support of the $2$-chain obtained by producing the Seifert dual spanning tree with the BFS$_1$ strategy. (c) The support of the $2$-chain obtained by producing the Seifert dual spanning tree with the BFS$_2$ technique.} \label{fig:scherk}
\end{figure}

\begin{figure}[!htb]
\centering
 \includegraphics[width=\textwidth]{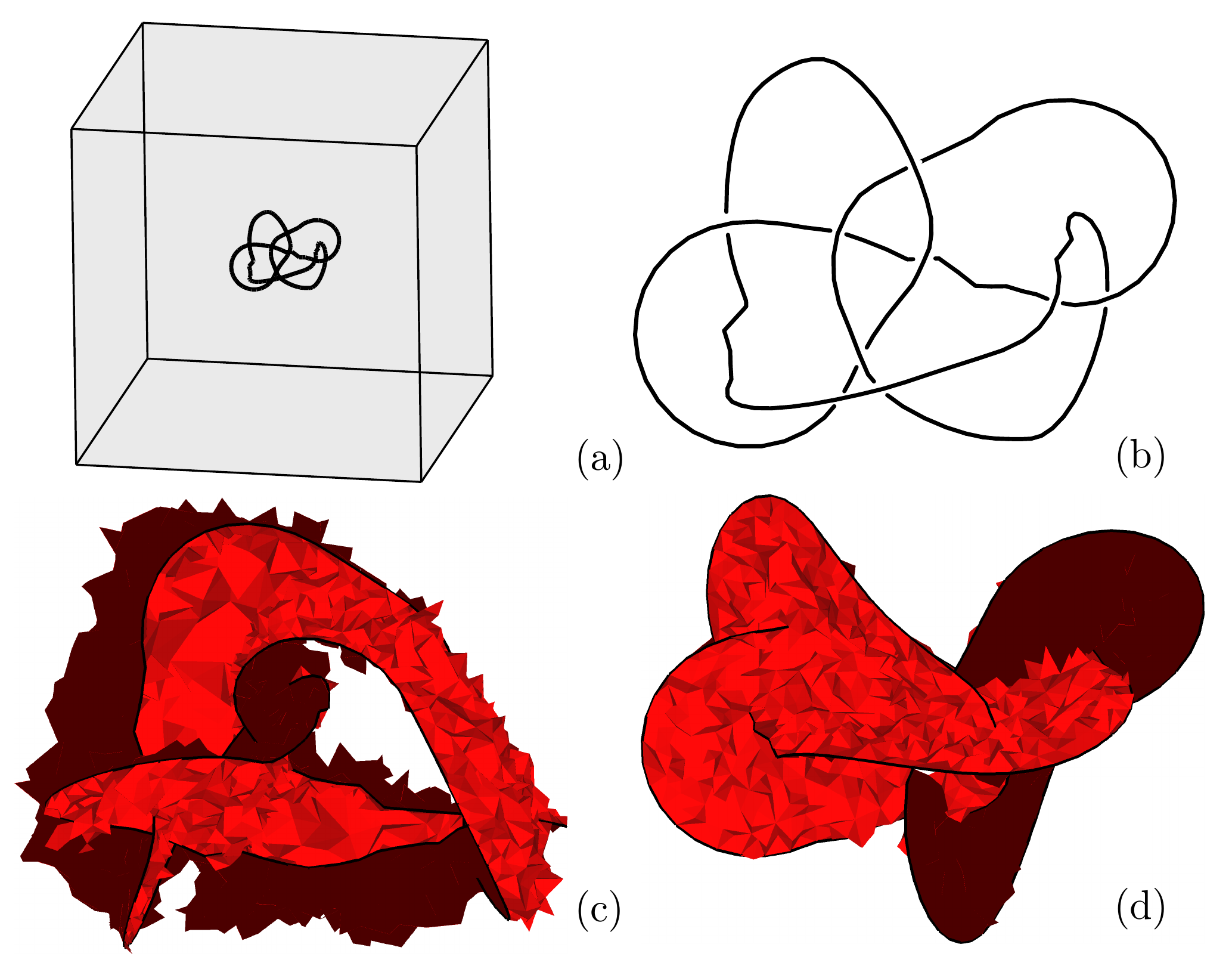}
  \caption{(a) The support of the $1$-boundary $\gamma$ is a $8_{21}$ knot placed inside a box outlined in the picture. (b) A zoom on $\gamma$. (c) The support of the $2$-chain obtained by producing the Seifert dual spanning tree with the BFS$_1$ strategy. (d) The support of the $2$-chain obtained by producing the Seifert dual spanning tree with the BFS$_2$ technique.} \label{fig:knot}
\end{figure}

The second technique, more closer to the philosophy of this paper and denoted by BFS$_2$, constructs the Seifert dual spanning tree $\cal B'$ as follows:
\begin{enumerate}
  \item Build a BFS spanning tree on each graph $\mathcal{A}_i'$ (not required in practice).
  \item Build a maximal plug-set $J$. That is, for each tetrahedron with at least one face in $F_{\partial}$, add exactly one plug induced by one of its faces in $F_{\partial}$.
  \item Form a tree in $(V',E')$ with the BFS strategy, by using all tetrahedra with at least one face in $F_{\partial}$ as root.
  \item If $\partial \Omega$ has more than one connected component, the preceding steps return a forest. To obtain a spanning tree of $\mathcal{A}'$, one may run the Kruskal algorithm \cite{cormen} starting from the forest already constructed.
\end{enumerate}

A possible maximal plug-set for the toy problem is represented in Figure \ref{fig:trivial}d. In the same picture, the dotted dual edges represent the plugs induced by corner faces whose plugs do not belong to the maximal plug-set $J$. The tree extended to the interior of the domain by running the BFS algorithm is represented in Figure \ref{fig:trivial}e. By running the elimination algorithm Alg. 1, one obtains the $2$-chain $S_{BFS_2}$, whose support is represented in Figure \ref{fig:trivial}f. In both cases, the obtained surfaces are non self-intersecting and $S_{BFS_2}$ is minimal.

In what follows, we present results for four more complicated benchmark problems.

We first consider a different toy problem in which $\gamma$ is the $1$-boundary of the cube represented in Figure \ref{fig:scherk}a by thicker edges. Figures \ref{fig:scherk}b and \ref{fig:scherk}c illustrate the support of the two $2$-chains $S_{BFS_1}$ and $S_{BFS_2}$ obtained by the BFS$_1$ and BFS$_2$ techniques, respectively.

Then, we take $\gamma$ as the non-trivial knot $8_{21}$ inside a cube, see Figure \ref{fig:knot}a (see also \cite[p. 394]{ROL76}). Figure \ref{fig:knot}b represents a zoom on $\gamma$. Figures \ref{fig:knot}c and \ref{fig:knot}d illustrate the support of the two $2$-chains $S_{BFS_1}$ and $S_{BFS_2}$ obtained by the BFS$_1$ and BFS$_2$ techniques, respectively.

As a third benchmark, we consider $\gamma$ as the Hopf link inside a cube, see Figure \ref{fig:hopf}a. The reader observes that the support of $\gamma$ has two connected components. Figure \ref{fig:hopf}b represents a zoom on $\gamma$. Figures \ref{fig:hopf}c and \ref{fig:hopf}d show the support of the two $2$-chains $S_{BFS_1}$ and $S_{BFS_2}$ obtained by the BFS$_1$ and BFS$_2$ techniques, respectively.

\begin{figure}[!htb]
\centering
 \includegraphics[width=\textwidth]{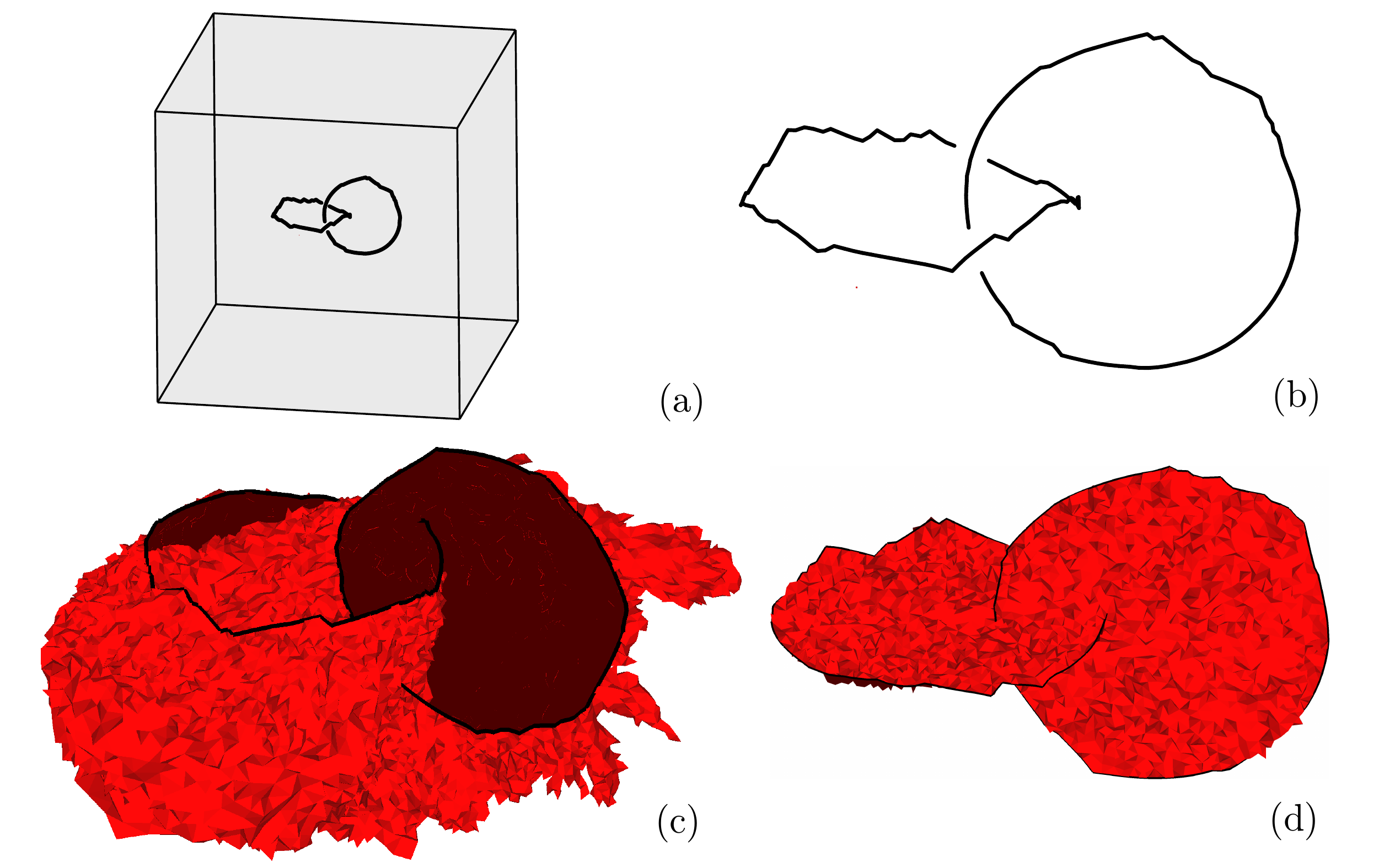}
  \caption{(a) The $1$-boundary $\gamma$ is a Hopf link placed inside a cube. (b) A zoom on $\gamma$. (c) The support of the $2$-chain obtained by producing the Seifert dual spanning tree with the BFS$_1$ strategy. (d) The support of the $2$-chain obtained by producing the Seifert dual spanning tree with the BFS$_2$ technique.} \label{fig:hopf}
\end{figure}

As a final example, we take $\gamma$ as a pair of disjoint circumferences placed in the boundary of a toric shell; namely, the difference of two coaxial solid tori, see Figure \ref{fig:nontrivial}a. Differently from the preceding cases, the computational domain; namely, the toric shell, is homologically non-trivial. Figures \ref{fig:nontrivial}b and \ref{fig:nontrivial}c illustrate the support of the two $2$-chains $S_{BFS_1}$ and $S_{BFS_2}$ obtained by the BFS$_1$ and BFS$_2$ techniques, respectively.

\begin{figure}[!htb]
\centering
 \includegraphics[width=\textwidth]{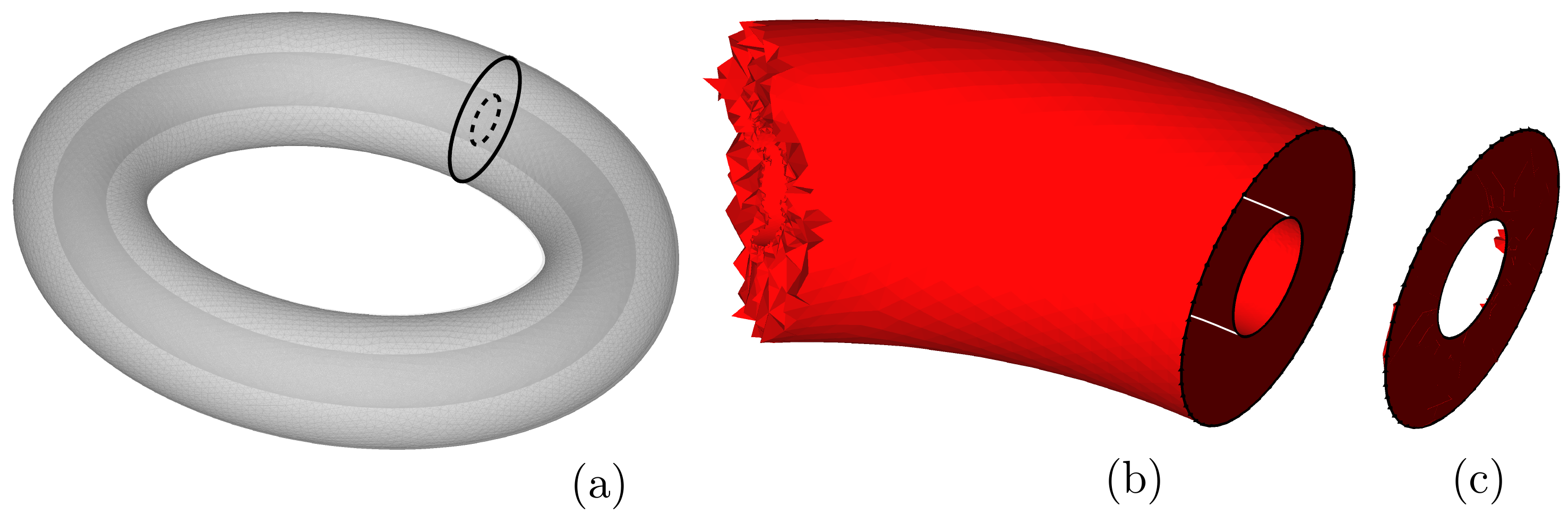}
  \caption{(a) The support of the $1$-boundary $\gamma$ is a pair of disjoint circumferences, outlined in the picture, placed on the boundary of a toric shell (namely, the difference between two coaxial solid tori). (b) The support of the $2$-chain obtained by producing the Seifert dual spanning tree with the BFS$_1$ strategy. (c) The support of the $2$-chain obtained by producing the Seifert dual spanning tree with the BFS$_2$ strategy.} \label{fig:nontrivial}
\end{figure}

The information about the number of geometric elements of the triangulation $\T$ and of the edges belonging to the support of the $1$-boundary $\gamma$ are stored in Table \ref{tab:card}. Table \ref{tab:res} shows the number of faces contained in the support of the $2$-chains obtained by the BFS$_1$ and BFS$_2$ techniques, together with the time (in milliseconds) required to obtain them. In Table \ref{tab:res}, it is also stated whether the support of the $2$-chains is self-intersecting or not.

After a considerable number of numerical experiments, we notice that the elimination algorithm Alg. 1 is able to construct the homological Seifert surface without the computation of any linking number. This happens also when the domain is not homologically trivial. Therefore, as anticipated, there is no need to compute a spanning tree of each graph $\mathcal{A}_i'$ and even to consider the dual graph $(V'_{\partial},E'_{\partial})$ on the boundary of $\Omega$. In fact, in the elimi\-nation step {\it 2.(b)}, only $\cal N' \cap \cal E'$ is used. The complete knowledge of $\cal N'$; namely, the construction of $\B_i'$ for every $i \in \{0,1,2,\dots,p\}$, is required just in the direct computation step. We do not have any explanation of this surprising feature of the algorithm yet. We also note that heuristically; namely, in all tested cases, the BFS$_2$ approach provides homological Seifert surfaces with strongly reduced support w.r.t. the BFS$_1$ technique.

Finally, we remark that when many homological Seifert surfaces are required on the same triangulation, Alg. 1 can be vectorialized in such a way that all surfaces are generated at once.

\begin{table}
  \centering
\rowcolors{2}{gray!25}{white}
\begin{tabular}{cccccc}
\rowcolor{gray!50}
Name & Tetrahedra & Faces & Edges & Vertices & $\mathrm{card}|\gamma|$ \\
Toy problem & 48 & 120 & 98 & 27 & 8\\
Toy problem $2$ & 479,435 & 973,963 & 583,183 & 88,656 & 341\\
$8_{21}$ knot & 87,221 & 175,317 & 102,212 & 14,117 & 170\\
Hopf link & 800,020 & 1,600,537 & 937,631 & 137,115 & 235\\
Toric shell & 1,851,494 & 3,871,379 & 2,419,350 & 399,465 & 176
\end{tabular}
 \caption{The number of geometric elements of the triangulation and of the edges belonging to the support of the $1$-boundary $\gamma$.}\label{tab:card}
\end{table}

\begin{table}
  \centering
\rowcolors{2}{gray!25}{white}
\begin{tabular}{ccccccc}
\rowcolor{gray!50}
Name & $\mathrm{card}|S_{BFS_1}|$ & Time$_{BFS_1}$ & Self-inters. & $\mathrm{card}|S_{BFS_2}|$ & Time$_{BFS_2}$ & Self-inters. \\
Toy problem & 24 & 2 & No & 8 & 1 & No\\
Toy problem $2$ & 15,089 & 220 & No & 15,023 & 233 & No\\
$8_{21}$ knot & 4188 & 38 & Yes & 2663 & 37 & Yes\\
Hopf link & 15,871 & 378 & Yes & 4841 & 407 & Yes\\
Toric shell & 46,786 & 986 & No & 1662 & 961 & No
\end{tabular}
 \caption{The number of faces belonging to the support $|S|$ of the homological Seifert surface $S$ and the time required (in milliseconds) for its generation by the proposed elimination algorithm, making use of the two different strategies for constructing a Seifert dual spanning tree. It is also mentioned whether the obtained surface is self-intersecting or not.}\label{tab:res}
\end{table}


\section*{Acknowledgements}
This work started during the fourth author stay at the Centro Internazionale per la Ricerca Matematica (CIRM), Fondazione Bruno Kessler (FBK), Trento, Italy as a Visiting Professor from March 3rd to March 29th in 2013. We thank Professor Marco Andreatta for his hospitality at CIRM. This work was finalized during the fourth author stay at University of Trento in July 2014.

\bibliographystyle{siam}
\bibliography{cicli}

\end{document}